\documentclass[11pt]{article}
\usepackage[plain,vlined]{algorithm2e}
\usepackage{amsmath}
\usepackage{amssymb}
\usepackage{amsthm}
\usepackage[english]{babel}
\usepackage{calc}
\usepackage{changepage}
\usepackage{cite}
\usepackage{color}
\usepackage{comment}
\usepackage{dsfont}
\usepackage[thinc]{esdiff}
\usepackage{enumitem}
\usepackage{fancyhdr}
\usepackage[T1]{fontenc}
\usepackage[margin=0.8in,letterpaper,top=2cm, bottom=2cm]{geometry} % decreases margins
\usepackage{graphicx} 
\usepackage[final]{hyperref} 
\usepackage[utf8]{inputenc}
\usepackage{latexsym}
\usepackage{longtable}
\usepackage{lmodern}
\usepackage{mathrsfs}
\usepackage{mathtools}
\usepackage{mleftright}
\usepackage{multicol}
\usepackage{subcaption}
\usepackage{tikz}
\usetikzlibrary{arrows,automata,positioning}
\usetikzlibrary{arrows.meta}
\usetikzlibrary{positioning}
\usepackage[normalem]{ulem}
\usepackage{url}
\usepackage{xypic}
\usepackage{xcolor}
\usepackage{xspace}
\usepackage{hyperref}
\newtheorem{theorem}{Theorem}[section]
\newtheorem{corollary}[theorem]{Corollary}
\newtheorem{definition}[theorem]{Definition}
\newtheorem{prop}[theorem]{Proposition}
\newtheorem{lemma}[theorem]{Lemma}
\newtheorem{remark}[theorem]{Remark}
\newtheorem{example}[theorem]{Example}

\usepackage{xcolor}

\newcommand{\dom}{\text{dom}}

\newcommand{\cP}{\mathcal{P}}
\newcommand{\cE}{\mathcal{E}}

\newcommand{\cG}{\mathcal{G}}

\newcommand{\cF}{\mathcal{F}}

\newcommand{\R}{\mathbb{R}}
\newcommand{\N}{\mathbb{N}}
\newcommand{\CAT}{\mathrm{CAT}}
\newcommand{\vertiii}[1]{{\left\vert\kern-0.25ex\left\vert\kern-0.25ex\left\vert #1 
		\right\vert\kern-0.25ex\right\vert\kern-0.25ex\right\vert}}

\newcommand{\proj}{\mathscr{P}}

\DeclareMathOperator{\Var}{Var}

\DeclareMathOperator{\prox}{prox}

\DeclareMathOperator*{\argmin}{arg\,min}
\DeclareMathOperator{\dist}{dist}

\theoremstyle{plain}

\theoremstyle{definition}
\newtheorem*{definition*}{Definition}
\newtheorem*{remark*}{Remark}

\newcommand{\eps}{\varepsilon}

\renewcommand{\d}{\mathrm{d}}

\newcommand{\supp}{\text{\rm supp}}

\newcommand{\Ent}{{\rm Ent}}
\renewcommand{\P}{\mathcal{P}}

\newcommand{\simone}[1]{{\color{blue}[Simo:   #1]}}
\title{\textbf{On the stability of proximal operators in Wasserstein spaces under different notions of convexity}}
\author{
	Simone Di Marino\footnotemark[1]\thanks{Dipartimento di Matematica, Università di Genova, Via Dodecaneso 35, 16146 Genova, Italy.}
	\thanks{\texttt{simone.dimarino@unige.it}} \and 
Sara Farinelli\footnotemark[1]
	\thanks{\texttt{sara.farinelli@edu.unige.it}}\and
	Emanuele Naldi\footnotemark[1]
    \thanks{\texttt{emanuele.naldi@edu.unige.it}} 
	}

\date{\today}

\begin{document}
\maketitle

\begin{abstract}
\noindent The proximal operator  is a fundamental tool in variational analysis and optimization. In the setting of a Hilbert space, given a proper, lower semicontinuous convex functional, its proximal operator is non-expansive, that is, $1$-Lipschitz continuous. In the Wasserstein setting, the contraction properties of this operator have been investigated from different perspectives by Carlen and Craig and  by Adve and Mészáros, among others, and are not completely understood. In this paper, we study the stability properties of proximal maps, with a particular focus on non-expansivity, under various notions of convexity of the functional that can be considered in the Wasserstein space.
\end{abstract}
\tableofcontents

%%%%%%%%%%%%%%%%%%%%%%%%%%%%%%%%%%%%INTRO 
%%%%%%%%%%%%%%%%%%%%%%%%%%%%%%%%
\section*{Introduction}

The proximal operator is a fundamental tool in variational analysis and optimization. In the setting of a Hilbert space $\mathcal{H}$, given a proper, lower semicontinuous, and convex functional $\cF: \mathcal{H} \to (-\infty, +\infty]$, the proximal map, defined by
\begin{equation}
    \prox_{\tau \cF}(x) \coloneqq \argmin_{y \in \mathcal{H}} \left\{ \cF(y) + \frac{1}{2\tau}\|x-y\|^2 \right\},
\end{equation}
enjoys several strong regularity properties. Key among these is non-expansivity: the map is $1$-Lipschitz continuous, satisfying $\|\prox_{\tau \cF}(x) - \prox_{\tau \cF}(y)\|_{\mathcal{H}} \leq \|x-y\|_{\mathcal{H}}$ for all $x, y \in \mathcal{H}$. {This stability is one of the cornerstones for proving  the convergence of proximal algorithms and discrete approximations of gradient flows.}

When moving from general Hilbert spaces to the Wasserstein space of probability measures $\cP_2(\R^d)$ endowed with the $W_2$ metric, the situation becomes more intricate. The Wasserstein proximal map, which drives the Jordan-Kinderlehrer-Otto scheme \cite{JKO, AGS08}, is defined as
\begin{equation}
    \prox_{\tau \cF}^W(\mu) \coloneqq \argmin_{\nu \in \cP_2(\R^d)} \left\{ \cF(\nu) + \frac{1}{2\tau}W_2^2(\mu,\nu) \right\}.
\end{equation}
On the one hand, in the one-dimensional case, it is well-known that the $2$-Wasserstein space $(\cP_2(\R), W_2)$ is isometric to a convex subset of the Hilbert space $L^2([0,1];\R)$ and so is a $\CAT(0)$ space (see \cite{Ot}). By exploiting this geometric property, it follows that the proximal map of geodesically convex functionals is non-expansive (see \cite{Jo}). On the other hand, for $d > 1$, the space $(\cP_2(\R), W_2)$ is no longer $\CAT(0)$ (see \cite{Ot}), and whether the proximal operator of a geodesically convex functional, when single-valued, retains this $1$-Lipschitz property is a notoriously open question.  This has been explicitly stated as an open problem in several seminal works, notably by Ambrosio, Gigli, and Savaré \cite[Remark 4.0.5 (c)]{AGS08}, Carlen and Craig \cite{CarlenCraig2013}, and Craig \cite{Craig2016}, who notes that ``it is unknown if such a contraction holds in the Wasserstein case.'' While true non-expansivity   remains unproven, some partial results exist. ``Almost'' contraction inequalities have been established for functionals which are convex along generalized geodesics by incorporating the variation of the functional itself into a modified distance \cite[Lemma 4.2.4]{AGS08}, \cite[Theorem 1.3]{CarlenCraig2013}: they prove that $\Delta_{i,\tau} (\prox_{\tau \cF}^W(\mu), \prox_{\tau \cF}^W(\nu)) \leq  \Delta_{i,\tau}(\mu, \nu) $, where 
\begin{equation}\label{eqn:DeltaAGS} \Delta_{1,\tau}(\mu, \nu) =  W_2^2(\mu, \nu) + \tau (\cF(\mu)+\cF(\nu)),
\end{equation}
\begin{equation}\label{eqn:DeltaCC}
\Delta_{2,\tau}(\mu, \nu) =  W_2^2(\mu, \nu) + \tau^2 (|\nabla^{-} \cF|^2(\mu)+|\nabla^{-} \cF|^2(\nu)).
\end{equation}
While these estimates, when used iteratively, provide non-expansivity   of the gradient flow in the regime $\tau \to 0$, they degenerate when the two measures are close. %in particular they do not privide stability informations in this case. 
The regularity of $\prox_{\tau \cF}^W$ is closely related to the specific notion of convexity satisfied by the functional $\cF$ and to recover full nonexpansivity, one must require stronger, more restrictive notions of convexity.\\

\begin{comment}
The objective of this paper is to understand the stability properties of the proximal map, with a particular focus on non-expansivity  , under different notions of convexity of the functional. Here we collect the main results of our work:
\end{comment}

In light of these considerations, the objective of this paper is to understand the stability properties of the proximal map, with a particular focus on non-expansivity, under different notions of convexity of the functional. We underline that every main result of the manuscript remains valid in $\P_2(\mathcal{H})$ for a Hilbert space $\mathcal{H}$. We decided to focus the exposition on $\P_2(\R^d)$ for simplicity and because it is the setting of primary interest
in applications. Our main contributions can be summarized as follows:

\begin{itemize}
    \item we provide a proof that for totally convex functionals,  the proximal map is non-expansive (this fact was already mentioned in \cite{DiNaVi}). This class of functionals which includes potential energies and interaction energies, is however rather rigid; in fact, using a monotonicity property with respect to the convex order, we show that neither are internal energy functionals totally convex, nor can they be approximated by totally convex functionals;

    \item we explore the concept of weak non-expansivity, {introduced in \cite{adve2020nonexpansiveness}}, which holds for every geodesically convex functional: thanks to it we can prove that if one of the measures is $\delta$-close to a Dirac delta  or a minimizer of the functional then the proximal map is $(1+\delta)$-Lipschitz;
    \item we introduce the new concept of convexity along \emph{2-base generalized geodesics}: under this assumption the proximal map is non-expansive;
    \item we extend a result of Roudneff-Chupin, proving that for any functional convex along generalized geodesics, the proximal map is locally $\frac 12$-H\"{o}lder.
\end{itemize}

While the study of non-expansiveness is  naturally motivated by its  role in the convergence of proximal algorithms, different notions of stability are equally relevant; for instance, Lipschitz continuity is a key assumption in the convergence of the method presented in \cite{Fe}. 

The paper is organized as follows. In Section~\ref{sec:prelim}, we introduce the main notation and concepts used throughout the paper, specifically the definitions of convexity in the Wasserstein space, the proximal operator and the convex order of measures.

\begin{comment}
In Section~\ref{sec:totallyconvex} we investigate the class of \emph{totally convex} functionals, recently introduced by Cavagnari, Savaré, and Sodini \cite{CSS_2023} (see also \cite{PiSa}). 
\end{comment}
{In Section~\ref{sec:totallyconvex} we investigate the class of \emph{totally convex} functionals, also called functionals convex along  \emph{acceleration-free curves}: a notion of convexity recently introduced independently in \cite{CSS_2023} and \cite{Pa}.}
A functional is totally convex if it is convex along any {acceleration-free curve that is a curve induced by linearly interpolating along the two variables of a transport plan}: this condition is strictly stronger than classical notions of convexity, namely standard convexity and convexity along generalized geodesics. It has already been observed that this strengthened convexity assumption can bring many Hilbert space properties to the convex functional. The recent development of optimization methods exploiting total convexity
\cite{Tanaka2023,BonetDC2026} further motivates the study of this class of functionals. In Section~\ref{seubsec:totconv-ne} we prove that the Wasserstein proximal map of totally convex functionals is non-expansive (1-Lipschitz). In Section~\ref{subsec:totconv-mon} we prove that certain functionals defined via the convex order are totally convex, and, more importantly, that every totally convex functional is monotonically increasing with respect to the convex order $\preccurlyeq_C$, i.e., if $\mu \preccurlyeq_C \nu$, then $\cF(\mu) \leq \cF(\nu)$. 
This monotonicity imposes severe structural constraints on the class of totally convex functionals (for example forcing their minimizers to be Dirac masses): we use these properties to prove that internal energy functionals are not totally convex and cannot be approximated by totally convex functionals. In Section~\ref{subsec:nonconvmoreau} we apply the previous results to find a class of geodesically convex functionals (namely those for which delta Diracs are not in their domain) with a non-convex Moreau envelope.

In Section~\ref{sec:weaklyconvex}, we turn our attention to the exploration of less restrictive assumptions that can guarantee generalized forms of non-expansivity.  In the first part of the section, we focus on weaker convexity conditions, showing that generalized forms of non-expansivity   can still be ensured without requiring total convexity.
 We introduce a notion of weak non-expansivity   (generalizing the definition in \cite{adve2020nonexpansiveness}): an operator $A: \mathcal{P}_2(\mathbb{R}^d) \to \mathcal{P}_2(\mathbb{R}^d)$ is said to be weakly non-expansive if for every $\mu,\nu \in \mathcal{P}_2(\mathbb{R}^d)$ there exists a nonempty subset $\Gamma_A(\mu,\nu) \subseteq \Gamma(\mu, \nu)$ such that 
$$W^2_2(A(\mu), A(\nu)) \leq \inf \left\{ \int_{\mathbb{R}^d \times \mathbb{R}^d} \|x-y\|^2 \, d\gamma(x,y) \; : \; \gamma \in \Gamma_A(\mu, \nu)\right\}.$$
We then prove that $\prox_{\tau \mathcal{F}}^W$ is weakly non-expansive when $\mathcal{F}$ satisfies suitable convexity assumptions. We present two possible choices for $\Gamma_A(\mu,\nu)$: the first assumes geodesic convexity (in Section~\ref{subsec:weak-geo}), while the second relies on convexity along generalized geodesics (in Section~\ref{subsec:weak-ggeo}). Both constructions are based on the concept of \emph{$3$-geodesic plans}, a key notion that we further exploit in the subsequent section.
\begin{comment}
In Section~\ref{sec:weaklyconvex} we explore other less restrictive assumptions which can guarantee some kind of non-exapansivity. We first introduce a notion of weak non-expansivity   (generalizing the definition in \cite{adve2020nonexpansiveness} where the authors were interested in the $W_2$-projection into the set of measures with density $\rho \leq 1$): an operator $A: \P_2(\R^d) \to \P_2(\R^d)$ is said weakly non-expansive if for every $\mu,\nu \in \P_2(\R^d)$ there exists $\Gamma_A(\mu,\nu) \subseteq \Gamma(\mu, \nu)$ nonempty such that 
$$ W^2_2(A(\mu), A(\nu)) \leq  \inf  \left\{ \int\|x-y\|^2 \, d\gamma \; : \; \gamma \in \Gamma_A(\mu, \nu)\right\}.$${We then prove that $\prox_{\tau \cF}^W$ is weakly non-expansive when $\cF$ is geodesically convex. We show two possible choices for $\Gamma_A(\mu,\nu)$, one in the case of geodesic convexity and one in the case of convexity along generalized geodesics, both based on a concept of \emph{$3$-geodesic plans}.}
\end{comment} 
In Section~\ref{subsec:two base}, we define a new notion of \emph{convexity along $(k-1)$-base generalized geodesics} (which are curves associated to $k$-geodesics plans), which generalizes the concepts of geodesic convexity, recovered for $k=1$, and convexity along generalized geodesics, obtained for $k=2$. We prove that the proximal operator of a functional that is convex along $2$-base generalized geodesics is non-expansive. We then discuss how large this class of functionals is. In the second part, we impose specific assumptions directly on the measures. In Section~\ref{subsec:almostlip}, we first analyze the cases where either $\mu =\delta_x$ or $\mu \in \argmin \mathcal{F}$. Under these conditions, the non-expansivity   of the proximal operator holds (see \cite{adve2020nonexpansiveness} and \cite{CarlenCraig2013}, respectively). We show that a quantitative relaxation of these statements can be established, yielding a conditional $(1+\delta)$-Lipschitz estimate. We conclude in Section~\ref{subsec:affine} by proving that non-expansivity holds if the optimal map between $\mu$ and $\nu$ is a dilation.

%pre sez 3
\begin{comment}
In the last section, Section \ref{sec:loc hol}, of this paper, we address the quantitative non conditional continuity of the proximal operator. Notice in fact that one of the drawbacks of \eqref{eqn:DeltaCC} and \eqref{eqn:DeltaAGS} is than
when $W_2(\mu,\nu)  \to 0$ the estimate does not go to zero, even though it is known that the proximal operator is continuous \cite[Lemma~4.1.2~(i)]{AGS08}. 
The other drawback is that they make sense respectively
only on the domain of $\cF$ and on the domain of $\nabla \cF$ and  This provides no information of course for the case when the measures are outside the domain of $\cF$. % is the indicator function  of a geodesically convex set, and so the proximal operator  becomes a projection operator.
In this case, in specific scenarios like density-constrained projections, the proximal map is known to be locally H\"older, though Lipschitzianity remains unknown \cite[Corollary 5.3, Remark 5.1]{DePhilippis2016}. We address both drawbacks: in Theorem~\ref{thm:hold-cont} we prove that the proximal map of any functional that is convex along generalized geodesics is locally $\frac{1}{2}$-H\"older continuous on the whole $\P_2(\R^d)$.\\
\end{comment}

In Section \ref{sec:loc hol}, the last section of this paper, we address the quantitative unconditional continuity of the proximal operator.
Notice that the estimates \eqref{eqn:DeltaAGS} and \eqref{eqn:DeltaCC} fail to imply continuity as $W_2(\mu,\nu) \to 0$, even though the proximal operator is known to be continuous \cite[Lemma~4.1.2~(i)]{AGS08}. In addition, these estimates require the measures to belong to the domain of $\cF$ and to the domain of $\nabla \cF$, respectively. Consequently, they provide no information for the case when the measures are outside the domain of $\cF$. Quantitative continuity has previously been proven only in specific scenarios, such as density-constrained projections, where the proximal map is known to be locally H\"older continuous (although its Lipschitz continuity remains unknown \cite[Corollary 5.3, Remark 5.1]{DePhilippis2016}). Here, we address both of these drawbacks by showing that local $\frac{1}{2}$-H\"older continuity holds in full generality: the proximal map of any functional that is convex along generalized geodesics is locally $\frac{1}{2}$-H\"older continuous on the whole $\P_2(\R^d)$.\\

\section{Preliminaries}\label{sec:prelim}
\subsection*{Notations}
We summarize below the main notation adopted in this paper.
\begin{longtable}{@{} p{4.0cm} p{12.8cm} @{}}

$\|\cdot\|$ & Euclidean norm of $\R^d$; \\

$\mathcal{L}^d$ &  Lebesgue measure of $\R^d$;\\ 

$\cP(\R^d)$ & set of Borel probability measures on $\R^d$;\\

$\cP_2(\R^d)$ & set of  Borel probability measures with finite second moment (see Section \ref{sec:prel-was});\\

$\cP_2^{a.c.}(\R^d)$ & subset of $\cP_2(\R^d)$ of probabilities absolutely continuous with respect to  $\mathcal{L}^d$;\\

$\supp(\mu)$ &  support of $\mu$;\\

$M(\mu)$ &  barycenter of $\mu$, $\int_{\R^d}x\, d\mu(x)$ (see Section \ref{subsec:totconv-mon});\\

$M_2(\mu)$ &  second moment of $\mu$;\\

 $\Var(\mu)$ & variance of $\mu$, $\int_{\mathbb{R}^d} \|x - M(\mu)\|^2 \, d\mu(x)$ (see Section \ref{subsec:totconv-mon});\\
 
{ $\pi^i: (\R^d)^{k}\to \R^d$} & the projection on the $i$-th component for $i=1,\dots,k$;\\

{ $\pi^{x}: (\R^d)^{k}\to \R^d$} & projection on the  component associated to coordinate $x$ (see Section \ref{sec:weaklyconvex});\\

{ $\pi^{i,j}: (\R^d)^{k}\to (\R^d)^2$} & projection   $(x_1,\dots,x_i, \dots, x_j, \dots, x_k)\mapsto (x_i,x_j)$ for $i,\,j=1,\dots,k$;\\

{ $\pi^{x,\tilde x}: (\R^d)^{k}\to (\R^d)^2$} & projection   $(x,\tilde x,\tilde y, y)\mapsto (x,\tilde x)$ (see Section \ref{sec:weaklyconvex});\\
 $\pi_t^{i\to j}$ & the convex combination of  components $i$ and $j$, $\pi_t^{i\to j}:= (1-t)\pi^i+t\pi^j$;\\
 
  $\pi_t^{x\to \tilde x}$ & convex combination of  components associated to coordinates $x$ and $\tilde x$, $\pi_t^{x\to \tilde x}:= (1-t)\pi^x+t\pi^{\tilde x}$ (see Section \ref{sec:weaklyconvex});\\
 $\Gamma(\mu_1,\dots, \mu_k)$ & the set of transport plans having $i$-th marginal equal to $\mu_i$ for any $i=1, \dots, k$ (see Section \ref{sec:prel-was});\\
 
 $\Gamma_{o}(\mu,\nu)$ & the set of optimal transport plans between $\mu$ and $\nu$ (see Section \ref{sec:prel-was}); \\

$\proj_K$ & Wasserstein  projection on the set $K$ (see Section \ref{subsec:prox});\\
 $\pi_t$ & $t$-interpolation between the two components when the space is a product of two spaces;\\
 $D(\cF)$ & the domain of a functional $\cF$ (see Section \ref{sec:prel-was});\\
  $|\nabla^{-}\cF|$ & the metric slope of a functional $\cF$ defined on $\cP_2(\R^d)$ (see Section \ref{sec:loc hol});\\
 $\mathbf{1}_{K}$ & indicator function of the set $K$ (see \eqref{eqn:indicatorset}); \\
\end{longtable}

%--------------- Convexity in P_2 ---------

\subsection{Wasserstein space and  convexity }\label{sec:prel-was}
In this section we recall some basics on the Wasserstein space and   the classical notions of convexity for functionals in this setting that will be used throughout the paper.  A complete and detailed reference on this topic can be found in \cite{AGS08} (see also \cite{Sant}). We then recall a notion of convexity introduced more recently in \cite{CSS_2023}.\\

Let $\cP(\R^d)$ be the set of Borel probability measures on $\R^d$. Consider the complete and separable metric space $(\cP_2(\R^d), W_2)$, defined as $\cP_2(\R^d)\coloneqq \{\mu\in \cP(\R^d)\mid \int_{\R^d}\|x\|^2\d\, \mu<+\infty\}$, and for $\mu$ and $\nu$  in $\cP_2(\R^d)$, one has 
\begin{align}\label{eq:dwas}
W_2^2(\mu,\nu)\coloneqq \inf_{\gamma\in \Gamma(\mu,\nu)}\int_{\R^d\times \R^d}\|x-y\|^2\d \, \gamma, 
\end{align}
 with $\Gamma(\mu,\nu)\coloneqq \left\{\gamma \in \cP(\R^d\times\R^d)\mid \pi^1_{\#}\gamma= \mu, \, \pi^2_{\#}\gamma= \nu\right\}$ and $\pi^{i}: \R^d\times \R^d\to \R^d $, the projection on the $i$-th component, $\pi^i:(x_1,x_2)\mapsto x_i $ for $i=1,\, 2$. The set of $\gamma \in \Gamma(\mu,\nu)$ realizing the infimum in \eqref{eq:dwas} is $\Gamma_o(\mu,\nu)$. For completeness we define here, for $\mu_1,\dots, \mu_k\in \cP(\R^d)$, the set $\Gamma(\mu_1,\dots, \mu_k)\coloneqq \{\gamma\in \cP((\R^d)^k)\mid\pi^{i}_{\#}\gamma=\mu_i\,\, \forall\,\,  i=1,\dots, k\}$. \\
 $(\cP_2(\R^d), W_2)$ is a geodesic metric space. A curve $(\mu_t)_{t\in[0,1]}$ is constant speed geodesic between $\mu_0$ and $\mu_1$ if and only if there exists $\gamma\in \Gamma_o(\mu_0,\mu_1)$, such that $\mu_t= (\pi_t^{1\to 2})_{\#}\gamma$, with $\pi_t^{1\to 2}: \R^d\times \R^d\to \R^d$, $\pi_t^{1\to 2}:(x_1,x_2)\mapsto (1-t)x_1+tx_2$.\\

For $\cF:\cP_2(\R^d) \to (-\infty,+\infty]$, we denote by $D(\cF)\coloneqq\{\mu\in \cP_2(\R^d)\mid \cF(\mu)<+\infty\}$, the domain of $\cF$. We say that $\cF$ is proper if $D(\cF)\neq\emptyset$.
For functionals defined on $\cP_2(\R^d)$, the standard notion of convexity is that of geodesic convexity, which naturally extends the classical convexity in linear spaces to any geodesic metric space. %In the sequel we define a more general concept of $\lambda$-convexity; a function is then convex when it is $\lambda$-convex with $\lambda=0$.

\begin{definition}[Convexity along  geodesics]
	Let $\cF: \cP_2(\R^d) \to (-\infty,+\infty]$ be a proper functional and let $\lambda \in \R$. $\cF$ is $\lambda$-convex along  geodesics if for every $\mu_0,\mu_1 \in \cP_2(\R^d)$  there exists a  geodesic $(\mu_s)_{s\in[0,1]}$ between $\mu_0$ and $\mu_1$  along which
	\begin{equation}
		\cF(\mu_s)\leq (1-s)\cF(\mu_0)+s\cF(\mu_1) -\frac \lambda2  s(1-s)W_2^2(\mu_0,\mu_1) \quad \text{for every } s \in [0,1].
	\end{equation}
    $\cF$ is said to be convex along  geodesics if it is $\lambda$-convex along geodesics with $\lambda=0$.
\end{definition}
%We recall that a curve $(\mu_s)_{s\in [0,1]}\subset  \cP_2(\R^d)$ is a constant speed geodesic between $\mu_0$ and $\mu_1$ if and only if there exists $\gamma\in \Gamma_o(\mu_0,\mu_1)$ such that $\mu_s=(\pi_t^{1\to 2})_{\#}\gamma$. \\

A stronger notion of convexity, which has turned out to be of relevance in the context of minimizing movements in the Wasserstein spaces, is the convexity along a bigger class of curves, called generalized geodesics. 
\begin{definition}[Generalized geodesic]
	A generalized geodesic between $\mu_0,\mu_1\in \cP_2(\R^d)$ with base $\nu\in \cP_2(\R^d)$ is a curve $(\mu_t)_{t\in [0,1]}$ in $\cP_2(\R^d)$  defined by
	\[\mu_t = (\pi_t^{2\to 3})_{\#} \gamma \quad t\in [0,1],\]
    where $\gamma \in \Gamma(\nu,\mu_0,\mu_1)$ such that $\pi_{\#}^{1,2}\gamma \in \Gamma_{o}(\nu,\mu_0)$ and $\pi_{\#}^{1,3}\gamma \in \Gamma_{o}(\nu,\mu_1)$. Moreover $\pi^{2,3}_{\#}\gamma \in \Gamma(\mu_0,\mu_1)$ is called \emph{coupling associated} to the generalized geodesic $(\mu_t)_{t \in [0,1]}$.
	%where $\pi_t^{2\to 3}:= (1-t)\pi^2+t\pi^3$, $\gamma \in \Gamma(\nu,\mu_0,\mu_1)$, $\pi_{\#}^{1,2}\gamma \in \Gamma_{o}(\nu,\mu_0)$ and $\pi_{\#}^{1,3}\gamma \in \Gamma_{o}(\nu,\mu_1)$, with $\pi^{i}$, for $i=1,2,3$, the projection on the $i$-th component,  $\pi^{1,2}:(x,y,z)\to (x,y)$ and $\pi^{1,3}:(x,y,z)\to (x,z)$. Moreover $\pi^{2,3}_{\#}\gamma \in \Gamma(\mu_0,\mu_1)$ is called the \emph{coupling associated} to the generalized geodesic $(\mu_t)_{t \in [0,1]}$.
\end{definition}

\begin{definition}[Convexity along generalized geodesics]
	Let $\cG: \cP_2(\R^d) \to (-\infty,+\infty]$ be a proper functional. $\cG$ is $\lambda$-convex along outer (respectively inner) generalized geodesics if for every $\mu_0,\mu_1 \in \cP_2(\R^d)$ and $\nu \in \cP_2(\R^d)$ (respectively $\nu \in D(\cG)$) there exists a generalized geodesic $(\mu_s)_{s\in[0,1]}$ between $\mu_0$ and $\mu_1$ with base $\nu$, along which
	\begin{equation}\label{eq:convgg}
		\cG(\mu_s)\leq (1-s)\cG(\mu_0)+s\cG(\mu_1)  -\frac \lambda2 s(1-s) \int \| x- y\|^2 \,d \gamma \quad \text{for every } s \in [0,1],
	\end{equation}
    where $\gamma \in \Gamma(\mu_0,\mu_1)$ is a coupling associated to the generalized geodesic $(\mu_s)_s$. 
        When $\lambda=0$ we  say that $\cG$ is  convex along  generalized geodesics.
\end{definition}
Since, given $\mu_0$ and $\mu_1$, the generalized geodesics between $\mu_0$ and $\mu_1$ with base $\mu_0$ are the geodesics between $\mu_0$ and $\mu_1$, convexity along generalized geodesics implies convexity along geodesics. 

In the following we will use convexity along \emph{outer} generalized geodesics; for this reason, in what follows we shall simply refer to it as convexity along generalized geodesics.

In \cite{AGS08}, only convexity along \emph{inner} generalized geodesics is studied. However, most of the functionals considered in the literature that are convex along inner generalized geodesics are also convex along outer generalized geodesics: in particular, lower semicontinuous functionals with dense domain that are convex along inner generalized geodesics are also convex along outer generalized geodesics.
Nevertheless, the two definitions are not equivalent in general, as the following example shows.
To fix the notations, we define here the indicator function of a set $K$, $ \mathbf{1}_{K}$
\begin{align}\label{eqn:indicatorset}
 \mathbf{1}_{K}(\mu)\coloneqq
\begin{cases}
 0 \quad &\text{if } \mu\in K, \\ 
 +\infty &\text{if } \mu\notin K.
\end{cases}
 \end{align}
 We have that $\mathbf{1}_{K}$ is lower semicontinuous whenever $K$ is closed; moreover we say that $K$ is (geodesically, generalized geodesically or totally) convex if $\mathbf{1}_{K}$ is (geodesically, generalized geodesically or totally) convex.
 
 \begin{example}[Convexity w.r.t. outer or inner generalized geodesics is different] \label{ex:in-not out}
   % An example of a functional which is convex along inner generalized geodesics, but not convex along outer generalized geodesics is the following.
   
   Let $\P_{\frac{1}{2}}(\R^d)$ be the set of $\frac{1}{2}$-quantized probability measures, that is 
    \begin{align*}
    \P_{\frac{1}{2}}(\R^d)\coloneqq\left\{\frac{1}{2}(\delta_{x_1}+\delta_{x_2}) \text{ with } x_1,\, x_2\, \text{not necessarily distinct }\right\}.
    \end{align*}
    The indicator function of this set  $\mathbf{1}_{\P_{\frac{1}{2}}(\R^d)}$  is convex along inner generalized geodesics. This follows from the fact that for any pair $\mu_0,\,\mu_1\in \P_{\frac{1}{2}}(\R^d)$, and for any $\nu= \frac 12 (\delta_{x_1} + \delta_{x_2} )\in\P_{\frac{1}{2}}(\R^d)$, there exists a generalized geodesic $(\mu_s)_{s\in [0,1]}$, with base $\nu$ such that $\mu_s\in\P_{\frac{1}{2}}(\R^d) $ for any $s\in [0,1]$. %it is obvious if $x_1=x_2$, otherwise
   By  using the existence of a Monge map between uniformly distributed discrete measures (Birkhoff theorem) one can build a generalized geodesic of the Monge form $\mu_s=(T_s)_\# \nu\in \P_{\frac 12}(\R^d)$. 
    
    However $\P_{\frac 12}(\R^d)$ is not convex along outer generalized geodesics for $d>1$: it suffices to consider $\mu_0=\frac 12 (\delta_{2e_1}+\delta_{-2e_1})$, $\mu_1= \frac 12 (\delta_{2e_2}+\delta_{-2e_2}) $ and $\nu$ the uniform measure on the unit ball. Since $\nu \ll \mathcal{L}^d$, there exists a unique generalized geodesic $\mu_s$ with base $\nu$, for which
    $$\mu_{\frac 12}=\frac 14 ( \delta_{e_1+e_2}+\delta_{e_1-e_2}+\delta_{e_2-e_1}+\delta_{-e_1-e_2} ) \not \in \P_{\frac 12} (\R^d).$$
    Similarly, for every $n \in \N, \, n\geq2$, $\P_{\frac 1n} (\R^d)$ is also convex along inner generalized geodesics but not along outer generalized geodesics.
\end{example}
{We finally recall a stronger notion of convexity introduced independently and simultaneously in \cite[Definition 5.1]{CSS_2023} as total convexity and in \cite[Definition 3.1]{Pa} as convexity along acceleration-free curves. 
An {acceleration-free curve} or  curve induced by a coupling between two probability measures $\mu$ and $\nu$ (possibly equal), is a curve $(\mu_t)_{t\in [0,1]}$, of the form 
\begin{equation*}
 \mu_t\coloneqq (\pi^{1\to 2}_t)_{\#}\gamma,\quad  \text{with}\quad\gamma\in \Gamma(\mu,\nu).
\end{equation*}}
\begin{definition}[Total convexity]
	Let $\cG: \cP_2(\R^d) \to (-\infty,+\infty]$ be a proper functional. $\cG$ is totally $\lambda$-convex if for every $\mu_0,\mu_1 \in \cP_2(\R^d)$, $\cG$ is $\lambda$-convex along any curve induced by a coupling, that is, for any $\gamma \in \Gamma(\mu_0,\mu_1)$ and $\mu_s\coloneqq(\pi_s^{1\to2})_{\#}\gamma$ it holds
	\begin{equation}\label{eqn:totconv}
		\cG(\mu_s)\leq (1-s)\cG(\mu_0)+s\cG(\mu_1) - \frac{\lambda}{2} s(1-s)\int \|x-y\|^2\, d\gamma \quad \text{for every } s \in [0,1].
	\end{equation}
       % with $\mu_s\coloneqq(\pi_s^{1\to2})_{\#}\gamma$.
       $\cG$ is totally convex if it is totally $\lambda$-convex with $\lambda=0$.
\end{definition}
{Although we adopt the term total convexity in this work, we wish to emphasize that Parker's original terminology remains better descriptive of the nature of the curves underlying this concept.}

We want to remark that for every definition of convexity in this section we have
\begin{equation}\label{eqn:equiv_lambdaconv} \cG \text{ is $\lambda$-convex if and only if }\cG- \frac{\lambda}2 M_2 \text{ is convex}\end{equation}
where $M_2(\mu)=\int_{\R^d} \|x\|^2 \, d\mu$ is the second moment functional. Notice that \eqref{eqn:equiv_lambdaconv} is a similar characterization of $\lambda$-convex functions in $\R^d$, with the $1$-convex function $\frac 12M_2$ in $\P_2(\R^d)$ playing the role of the $1$-convex function $x \mapsto \frac 12 \|x\|^2$ in $\R^d$.
%Observe that the indicator function  of a set $K\subseteq \P_2(\R^d)$, $\mathbf{1}_{K}$,  is totally convex if and only if $K$  for any $\mu_1, \, \mu_2\in K$, and any interpolating curve $t\mapsto \mu_t$ induced by a coupling between them, we have  $\mu_t\in K$. 

%\textcolor{red}{Let $\pi^t: \R^d\times \R^d\to \R^d$ be the interpolation map,
%\begin{equation*}
%\pi^t:(x,y)\mapsto (1-t)x+ty.
%\end{equation*}}
%----------------------Proximal operator --------------------
\subsection{Proximal operator}\label{subsec:prox}
Let $\mathcal{F}:\cP_2(\R^d)\to (-\infty, +\infty]$ be proper and lower semicontinuous. The proximal operator associated to $\mathcal{F}$ of step $\tau>0$ is 
\begin{equation}\label{eq:prox}
    \prox_{\tau \cF}^W(\mu) \coloneqq \argmin_{\nu \in \cP_2(\R^d)} \left\{ \mathcal{F}(\nu) + \frac{1}{2\tau}W_2^2(\mu,\nu) \right\}.
\end{equation}
A particular case of proximal operator which has been investigated in the literature is that of projection operators. Let $K\subseteq \cP_2(\R^d)$, then 
%be a closed set and denote by $\mathbf{1}_{K}$ the indicator function of the $K$:
%\begin{align*}
%\mathbf{1}_{K}(\nu)=
%\begin{cases}
%0 &\text{if } \nu\in K,\\
%+\infty&\text{if not}.
%\end{cases}
%\end{align*}
%Then 
\begin{align}\label{eq:prox=proj}
\prox_{\tau \mathbf{1}_K}^W(\mu)=   \proj_K(\mu)
\end{align}
where  $\mathbf{1}_K$ is the indicator function of $K$ as in \eqref{eqn:indicatorset} and $\proj_K$ is the $W_2$-projection on the set $K$.

If $\mathcal{F}$ is lower semicontinuous for the weak convergence and there exists $\alpha,\, \beta>0$ and $\bar \mu\in \cP_2(\R^d)$ such that 
$\mathcal{F}(\mu)\geq -\frac{\alpha}2 W_2^2(\bar \mu, \mu)-\beta$ for any $\mu\in \cP_2(\R^d)$, then for $\tau <\frac{1}{\alpha}$ the proximal operator \eqref{eq:prox} is well defined in the sense that the set of minimizers in \eqref{eq:prox} is non empty, see for example \cite[Corollary~2.2.2]{AGS08}. In general this operator can be multivalued, but as soon as $\mathcal{F}$ is $\lambda$-convex and $\lambda \tau > - 1$ this cannot happen.
\begin{lemma}
Let $\mathcal{F}:\cP_2(\R^d)\to (-\infty, +\infty]$ be proper,  weakly lower semicontinuous and $\lambda$-convex along outer generalized geodesics. If $\lambda \tau >-1$ then for every $\mu \in \P_2(\R^d)$ we have that $\prox_{\tau \cF}^{W}(\mu)$ is single-valued and nonempty.
\end{lemma}

The proof of this fact when $\mu \in \overline{D(\cF)}$ for a $\lambda$-convex function along \emph{inner} generalized geodesics can be found in \cite[Lemma~4.1.2(i)]{AGS08}. The proof for any $\mu \in \P_2(\R^d)$ if $\cF$ is $\lambda$-convex function along \emph{outer} generalized geodesics follows verbatim the same proof but using the outer generalized geodesics. Moreover, as explained in the following remark (which follows a counterexample illustrated in \cite{DiNaVi}), convexity along outer generalized geodesics is needed in order to have uniqueness for a general $\mu$.
\begin{remark}
Let us consider $K=\P_{\frac 12} (\R^d)$ as in Example \ref{ex:in-not out}: there it is shown that $\textbf{1}_{K}$ is convex along inner generalized geodesic but is \emph{not} convex along outer generalized geodesics.
Let $\mu=\chi_B\mathcal{L}^d$ be the characteristic function of a ball $B$ with $\mathcal{L}^d(B)=1$. Then $\prox^{W}_{\tau \textbf{1}_{K} }(\mu)$ is not single-valued. In fact, for example, if $d=2$, 
$$\prox^{W}_{\tau \textbf{1}_{K} }(\mu)=\left\{\frac{1}{2}\delta_{R_{\theta}x}+\frac{1}{2}\delta_{R_{\theta}(-x)}\mid\theta \in [0,2\pi]\right\},$$
where $R_{\theta}$ is the rotation by angle $\theta$ and {$x$ is the barycenter of $\mu_{|\{x_1>0\}}$}. 
   % \textcolor{gray}{
	%Let $C$ be the set of measures composed by one or two deltas. Then $C$ is convex along inner generalized geodesics but $J_{\gamma \cG}$ is not Lipschitz. For example, given as $\mu$ the Gaussian with mean zero and the identity as covariance matrix, the projection onto $C$ is not unique.} 
    %\sara{Vogliamo formalizzare questo? ovvero il fatto che il grafico è chiuso?}
    %\textcolor{gray}{Moreover, taking Gaussians that go to round Gaussian in one direction or the other. They jump. The projection would then probably have a closed graph, but not Lipschitz...}
    
    Following \cite[Lemma~4.1.1]{AGS08} one could show that the graph of the multivalued operator $\mu \mapsto \prox^{W}_{\tau \textbf{1}_{K} }(\mu)$ is actually closed; however it is not continuous even in its uniqueness domain. In order to show it, it is sufficient to consider $\mu^\theta_{\eps}=\chi_{E_{\eps}^{\theta}} \mathcal{L}^d$, where $E_{\eps}^{\theta}$ are unit volume ellipses slightly stretched in the direction $\theta$; as $\eps \to 0$ for every $\theta$ we will have
    $$\prox_{\tau \textbf{1}_{K}}^{W}(\mu^{\theta}_{\eps}) \to \frac{1}{2}\delta_{R_{\theta}x}+\frac{1}{2}\delta_{R_{\theta}(-x)}, $$
    while having $W_2(\mu^{\theta}_{\eps}, \mu^{\theta'}_{\eps}) \lesssim \eps \to 0  $.
\end{remark}

%---------------- Conv order------------
\subsection{Convex order of probability measures}
Here we recall the notion of convex order, that is a notion of partial ordering on the set of probability measures $\cP(\R^d)$ and in particular in  $\cP_2(\R^d)$ (see e.g.\cite{Strassen65}).  

\begin{definition}[Convex order]
Given  $\mu,\nu\in\P_2(\R^d)$, $\mu$ and $\nu$ are in convex order $\mu\preccurlyeq_C\nu$ if $\int_{\R^d}f\d \mu\leq \int_{\R^d}f\d \nu$ for any convex function $f$. 
\end{definition}
An important characterization in terms of Markov kernels is the following Theorem, for which we refer e.g. to \cite[Theorem 2]{Strassen65}.
\begin{theorem}[Characterization of convex order]\label{thm:char_of_con_ord}
$\mu,\nu\in\P_2(\R^d)$ are in convex order $\mu\preccurlyeq_C\nu$ if and only if there exists a family  $\{\eta_{z}\}_{z\in \R^d}$ of probability measures of $\R^d$, such that 
\begin{itemize}
\item for any $B\subseteq \R^d$ Borel, $z\mapsto \eta_{z}(B)$ is $\mu$ measurable,  (the family is weakly Borel);
\item for $\mu$-a.e. any $z\in \R^d$, $M(\eta_z)=z$;
\item $\nu= \int_{\R^d}\eta_{z}\d \mu(z)$ i.e. for any $B\subseteq \R^d$, $\nu(B)= \int_{\R^d}\eta_{z}(B)\d \mu(z)$.
\end{itemize}
\end{theorem}
The characterization reads well in terms of random variables: $\mu \preccurlyeq_C \nu$ if there exist a probability space $(\Omega, \mathbb{P})$ and random variables $X,Y \in L^2(\mathbb{P};\R^d)$ whose laws are respectively $\mu$ and $\nu$, and 
$\mathbb{E} ( Y | X) =X$.
%%%%%%%%%%%%%%%%%%%%%%%%%%%%%%%%%%%%%%%%%%%%
%%%%%%%%%%%% Sect 1 %%%%%%%%%%%%%%%%%%%%%%%%
%%%%%%%%%%%%%%%%%%%%%%%%%%%%%%%%%%%%%%%%%%%%
\section{Totally convex functionals, convex order of measures and proximal operator}\label{sec:totallyconvex}

In this section we explore the totally convex functionals. First we prove, as a consequence of results in \cite{CSS_2023} and \cite{PiSa}, that for this class of functionals we have in fact non-expansivity   of the proximal operator.

Then we prove a rigidity result for totally convex functionals, namely that they are monotone with respect to the convex order, see Theorem~\ref{thm:mon-conv}. This rigidity will imply not only that many functionals $\cG$ which are convex with respect to generalized geodesics are not totally convex, but also that
\begin{itemize}
    \item $\prox_{\tau \cG}^{W}$ cannot be approximated by $\prox_{\tau \cG_n}^{W}$ for a sequence of $\cG_n$ totally convex functionals (Proposition~\ref{prop:prox_not_approx})
    \item $\cG$ cannot be approximated by totally convex functionals $\cG_n$.
\end{itemize}
This last point will allow us to prove that for a quite general class of $\cG$ convex with respect to generalized geodesics, we have that for every $\lambda$, $\cG^{\tau}$ is not even $\lambda$-convex along generalized geodesic for $\tau$ sufficiently small.

%\simone{introduzione}

\subsection{Non-expansivity of the proximal operator of totally convex functionals} \label{seubsec:totconv-ne}
\begin{prop}\label{prop:tot-conv-non exp}
Let $\cF:\cP_2(\R^d)\to (-\infty,+\infty]$ be proper, lower semicontinuous and totally convex. Then $\prox_{\tau \cF}^{W}(\mu):\cP_2(\R^d)\to \cP_2(\R^d)$ is non-expansive. 
\end{prop}
To prove the result, we use the Lagrangian lifting of $\cP_2(\R^d)$: for a reference see e.g. \cite{CSS_2023}. We refer to the notations of \cite[Section 2.3]{PiSa}.
Given $(\Omega,\mathbb{P})$ a probability space with $\mathbb{P}$ nonatomic, consider the space $\mathcal{H}\coloneqq L^{2}(\Omega,\mathbb{P};\R^d)$ and the map 
\begin{align*}
\iota:\, &\mathcal{H}\to \cP_2(\R^d),\\
&X\mapsto X_\#\mathbb{P},
\end{align*}
which satisfies $W_2(\iota(X_1),\iota (X_2))\leq \|X_1-X_2\|_{\mathcal{H}}$.

As stated in \cite[Lemma 2.7]{PiSa} this map is surjective. Moreover, for any pair $\mu_1,\,\mu_2\in \cP_2(\R^d)$ one has that 
\begin{align}\label{eq:was-hil}
W_2(\mu_1, \mu_2) = \inf_{\substack{X_1, X_2 \in \mathcal{H}\\ (X_i)_{\#}\mathbb{P} = \mu_i}}\|X_1-X_2\|_{\mathcal{H}}. 
\end{align}
Indeed by \cite[Lemma 2.7]{PiSa}, which guarantees that for any $\gamma\in \Gamma(\mu_1,\mu_2)$, there exists $X_{1,\gamma}, X_{2,\gamma}$, such that $(X_{1,\gamma}, X_{2,\gamma})_{\#}\mathbb{P}= \gamma$ and therefore $\int_{\R^d\times \R^d}|x-y|^2\gamma = \|X_{1,\gamma}- X_{2,\gamma}\|^2_{\mathcal{H}}$. \\
For any $\cF:\cP_2(\R^d)\to (-\infty,+\infty]$, we denote by $\hat{\cF}$ its Lagrangian lifting, that is 
\begin{align*}
\hat \cF:\, &\mathcal{H} \to (-\infty,+\infty],\\
&X\to \cF(X_{\#} \mathbb{P}). 
\end{align*}
\begin{lemma}
Let $\cF:\cP_2(\R^d)\to(-\infty,+\infty]$ be proper, lower semicontinuous and totally convex. Then $\prox_{\tau \cF}^W(X_{\#}\mathbb{P})=\prox^{L^2}_{\tau \hat \cF}(X)_{\#}\mathbb{P}$.
\end{lemma}
\begin{proof}
It follows directly from \cite[Proposition 5.4, point (1)]{CSS_2023} by recalling that the resolvent operator of $\hat{\mathcal{F}}$, $\mathrm{J}_{\tau}$ is $\prox^{L^2}_{\tau\hat{\mathcal{F}}}$.
\end{proof}
\begin{proof}[Proof of Proposition \ref{prop:tot-conv-non exp}]
Let $\mu_1,\mu_2\in \cP_2(\R^d)$. Let $X_1,\,X_2$ be such that $\mu_i=X_i\# \mathbb{P}$ for $i=1,2$. By the surjectivity of the map $\iota$, at least one pair exists. One has
\begin{align*}
&W_2(\prox_{\tau \cF}^{W}(\mu_1),\prox_{\tau \cF}^{W}(\mu_2))= W_2(\prox^{L^2}_{\tau \hat \cF}(X_1)_{\#}\mathbb{P},\prox^{L^2}_{\tau \hat \cF}(X_2)_{\#}\mathbb{P}) \\&\leq\|\prox^{L^2}_{\tau \hat \cF}(X_1)-\prox^{L^2}_{\tau \hat \cF}(X_2)\|\leq \|X_1-X_2\|_{\mathcal{H}},
\end{align*}
where the last inequality follows from the non-expansivity   of the proximal operator of any convex function in Hilbert spaces. Therefore, optimizing in $X_1$ and $X_2$ and using \eqref{eq:was-hil} we conclude
\begin{align*}
W_2(\text{prox}_{\tau \cF}^{W}(\mu_1), \text{prox}_{\tau \cF}^{W}(\mu_2)) \leq \inf_{\substack{X_1, X_2 : \\ X_i \# \mathbb{P} = \mu_i}} \|X_1- X_2\|_{\mathcal{H}}=W_2(\mu_1,\mu_2).  
\end{align*}
\end{proof}
\begin{corollary}
Let $K\subseteq \cP_2(\R^d)$ be a nonempty, closed, totally convex set, that is for any $\mu_1$ and $\mu_2$ in $K$, and any $\gamma\in \Gamma(\mu_1,\mu_2)$, $\mu_t\coloneqq (\pi_t)_{\#}\gamma\in K$ for any $t\in [0,1]$. Then $\proj_K$, the $W_2$-projection on $K$ is $1$-Lipschitz.
\end{corollary}

\begin{remark}\label{rem:totlambda} A similar statement holds for a totally $\lambda$-convex function $\cF$. Using the result in Hilbert spaces \cite[Proposition 3.3]{HoheiselLabordeOberman2020}, for $\tau>0$ such that $\lambda \tau > -1$ we have
\begin{equation}\label{eqn:nonexp-lambda}
    W_2( \prox_{\tau \cF}^{W} (\mu_1), \prox_{\tau \cF}^{W} (\mu_2)) \leq  \frac 1{1+\lambda \tau}W_2(\mu_1, \mu_2) \qquad \forall \mu_1, \mu_2 \in \P_2(\R^d).
\end{equation}
\end{remark}

%\sara{Esempi di insiemi totalmente convessi}

\subsection{Monotonicity of totally convex functionals with respect to convex order of measures}\label{subsec:totconv-mon}

We investigate here an interesting rigid property of totally convex functions, namely their monotonicity with respect to the convex order (Theorem~\ref{thm:mon-conv}): this will be our main tool to prove that the set of totally convex functions is pretty small and cannot be used to approximate convex functions along generalized geodesics in any sense. The results in this section are presented only for $\lambda=0$ because the statements are much clearer that way: obviously one can use \eqref{eqn:equiv_lambdaconv} to have the corresponding properties for totally $\lambda$-convex functions for $\lambda \neq 0$. 

During the preparation of the manuscript, we were made aware that the result has appeared recently in \cite{Be} (where they claim also to have a necessary and sufficient condition for being totally convex). Being our proof different and more constructive, we decided to keep it.

\subsubsection*{Monotonicity with respect to convex order}
We recall that for $\mu\in \cP_2(\R^d)$ the barycenter of $\mu$ is $M(\mu)\coloneqq\int_{\R^d}x\,\d \mu$ and the variance of $\mu$ is $\Var(\mu)\coloneqq \int_{\R^d}\|x-M(\mu)\|^2\, \d\mu$. 
\begin{prop}\label{prop:delta-bar}
Let $\cG: \cP_2(\R^d)\to (-\infty,+\infty]$ be proper, lower semicontinuous and totally convex. 
%totally {\color{red} $\lambda$-}convex
Then \[\cG(\delta_{M(\mu)})\leq \cG(\mu) \quad \text{for every} \quad \mu\in \cP_2(\R^d).\] 
\end{prop}
\begin{proof}
We construct a monotone decreasing sequence recursively defined. We fix $t\in (0,1)$ and define $(\mu_{n})_{n\in \mathbb{N}}$:  $\mu_0\coloneqq\mu$, 
\begin{align*}
\mu_{n+1}\coloneqq\pi^t_{\#}\gamma_n, \quad \text{with } \gamma_n\coloneqq\mu_n \otimes \mu_n, \qquad \text{for every } n\in \mathbb{N}.
\end{align*}
Then by Lemma~\ref{lem:decreasing_variance} $\mu_{n+1}\preccurlyeq_C \mu_{n}$, and $\mu_n\to \delta_{M(\mu)}$. %If $\mu=\delta_x$ for some $x\in \R^d$, then  $\mu_n=\delta_x$ for any $n\in \mathbb{N}$ and there is nothing to prove. Otherwise, 
In fact %, from Lemma\ref{lem:decreasing_variance} 
inductively we have $M(\mu_n)=M(\mu)$ for every $n\in \N$ and $\Var(\mu_n)=(1-2(1-t)t)^n\Var(\mu)$ so $\Var(\mu_n)\to 0$. Since we have $W_2^2(\mu_n,\delta_{M(\mu)})=W_2^2(\mu_n,\delta_{M(\mu_n)})= \Var(\mu_n)$, we can conclude that $\mu_n\to \delta_{M(\mu)}$. Moreover by total convexity $\cG(\mu_{n+1})\leq \cG(\mu_n)$. The result follows then by lower semicontinuity of $\cG$.  
\end{proof}
We observe that for any measure $\mu$, Jensen's inequality yields $\delta_{M(\mu)}\preccurlyeq_C\mu$. In fact, any Dirac measure is a minimal element with respect to the convex order. Moreover  for any $\mu$ the Dirac delta concentrated in the barycenter of $\mu$, serves  as the unique minimum among the set of lower bounds of $\mu$.

In general, the procedure developed herein to construct a sequence of measures starting from $\mu$ and monotonically decreasing in convex order to $\delta_{M(\mu)}$, can be naturally extended to construct a decreasing sequence 'from $\mu$ to $\nu$' with $\nu\preccurlyeq_C\mu$. The following theorem 
which asserts that totally convex functionals are monotonic with respect to the convex order of measures then follows.
\begin{theorem}\label{thm:mon-conv}
Let $\cG:\P_{2}(\R^d)\to (-\infty,+\infty]$ proper, lower semicontinuous and totally convex. Then for any $\mu,\nu\in\P_{2}(\R^d)$ satisfying $\mu\preccurlyeq_C\nu$,
\begin{align*}
\cG(\mu)\leq \cG(\nu).
\end{align*}
\end{theorem}
 We  present a constructive and direct proof of the result.  We begin with some technical lemmas that we need.
 \begin{lemma}\label{lem:decreasing_variance}
Let $\mu\in \cP_2(\R^d)$ and consider the product plan $\gamma = \mu \otimes \mu$. Then 
\[\pi^t_{\#} \gamma \preccurlyeq_C \mu, \qquad M(\pi^t_{\#}\gamma)=M(\mu), \quad \text{and} \quad \Var(\pi^t_{\#}\gamma)= (1-2t(1-t))\Var(\mu).\]
In particular, it holds $\Var(\pi^t_{\#}\gamma) \leq \Var(\mu)$ and for $t\in (0,1)$, $\Var(\pi^t_{\#}\gamma) = \Var(\mu)$ if and only if $\mu$ is a delta.
\end{lemma}
\begin{proof}
Notice that for every convex function $f$ we have
$$\int f d \pi^t_{\#}\gamma = \iint f(tx+(1-t)x') \d \mu \otimes \mu \leq \iint tf(x)+(1-t)f(x') \d \mu \otimes \mu = \int f \, \d \mu.$$
In particular using $f(x)= \pm x_i$ we get also $M(\pi^t_{\#}\gamma)=M(\mu)$. For the variance we have
%\begin{align*}
%M(\pi^t_{\#}\gamma)= \int_{\R^d} z \, \d\pi^t_{\#}\gamma(z) = \iint_{\R^d\times\R^d} \left((1-t)x+ty\right)\, \d\mu(x)\d\mu(y)= M(\mu).
%\end{align*}
%We have
\begin{align*}
\Var(\pi^t_{\#}\gamma) & = \int_{\R^d} \|z-M(\mu)\|^2 \, \d\pi^t_{\#}\gamma(z) = \iint_{\R^d\times\R^d} \|(1-t)x+ty - M(\mu)\|^2 \, \d\mu(x)\d\mu(y)\\
& = \iint_{\R^d\times\R^d} \|(1-t)(x-M(\mu))+t(y - M(\mu))\|^2 \, \d\mu(x)\d\mu(y)\\
& = (1-t)^2\int_{\R^d} \|x-M(\mu)\|^2 \, \d\mu(x)+t^2\int_{\R^d} \|y-M(\mu)\|^2 \, \d\mu(y)\\&=(1-2t(1-t))\Var(\mu).
\end{align*}
\begin{comment}
On the other hand 
\begin{align*}
\iint_{\R^d\times\R^d} \|x-y\|^2 \, \d\mu(x)\d\mu(y) & = \int_{\R^d} \|x\|^2 \, \d\mu(x) + \int_{\R^d} \|y\|^2 \, \d\mu(x)-\iint_{\R^d\times\R^d} 2\langle x , y \rangle \, \d\mu(x)\d\mu(y)\\
& = \int_{\R^d} \|x\|^2 \, \d\mu(x) + \int_{\R^d} \|y\|^2 \, \d\mu(x)-2\left\langle \int_{\R^d} x \, \d\mu(x), \int_{\R^d} y \, \d\mu(y) \right\rangle\\
& =  2\left(\int_{\R^d} \|x\|^2 \, \d\mu(x)-\langle M(\mu), M(\mu) \rangle\right)\\
& = 2\int_{\R^d} \|x-M(\mu)\|^2 \, \d\mu(x)=2 \Var(\mu). 
\end{align*}
\end{comment}
From the computations, we deduce that equality holds if and only if $\Var(\mu)=0$, that is $\mu= \delta_{x}$ for some $x\in \R^d$. 
\end{proof}

\begin{lemma}
Let $\nu$ be in $\P_{2}(\R^d)$, and $\mu\preccurlyeq_C\nu$, $\nu= \int_{\R^d}\eta_z \d \mu(z)$.  Then $\int_{\R^d}Var((\pi_{t})_{\#}(\eta_z\otimes \eta_z))\d \mu(z)\leq (1-2t(1-t))\int_{\R^d}Var(\eta_z)\d \mu(z)$ and $\int_{\R^d}Var(\eta_z)\d \mu(z)<+\infty$.
\end{lemma}
\begin{proof}
The first inequality follows from Lemma \ref{lem:decreasing_variance}. 
For the second 
\begin{align*}
&\int_{\R^d}Var(\eta_z)\d \mu(z)= \int_{\R^d}\int_{\R^d}|x-M(\eta_z)|^2\d \eta_z(x)\mu(z)\\
&=  \int_{\R^d}\int_{\R^d}|x|^{2}\d \eta_z(x)\d \mu(z)- \int_{\R^d}\int_{\R^d}|z|^2\d \eta_z(x)\d \mu(z)=  \left( M_2(\nu)-M_2(\mu)\right).
\end{align*}
\end{proof}
\begin{prop}\label{prop:mon-converg-sequence}
Let $\nu$ be in $\P_{2}(\R^d)$, and $\mu\preccurlyeq_C\nu$, $\nu= \int_{\R^d}\eta_z \d \mu(z)$. Fix $0<t<1$ and consider the sequence of probability measures defined recursively as follows: $\mu_0\coloneqq \nu$, 
\begin{align*}
\eta_{z}^{1} &\coloneqq (\pi^t)_{\#}(\eta_{z}\otimes\eta_{z}),
\qquad & \mu_1 &\coloneqq \int_{\R^d}\eta_{z}^{1}\,\d \mu(z),\\
\eta^{n}_{z} &\coloneqq (\pi^t)_{\#}(\eta^{n-1}_{z}\otimes\eta^{n-1}_{z}) \text{ for } n\geq 2
\qquad & \mu_n &\coloneqq \int_{\R^d}\eta_{z}^{n}\,\d \mu(z)  \text{ for } n\geq 2.
\end{align*}
Then  $\{\mu_{n}\}_{n\in \mathbb{N}}$  is decreasing for the convex order, that is $\mu_{n}\preccurlyeq_C \mu_{n-1}$ for any $n\in \mathbb{N}$ and  $\mu_{n}\to \mu$ in $(\P_2(\R^d),W_2)$.
\end{prop}
\begin{proof}
First we observe that 
\begin{align*}
\mu_n= \pi^t_{\#}\left(\int_{\R^d}\eta^{n-1}_{z}\otimes\eta^{n-1}_{z}\d\mu(z)\right), 
\end{align*}
and therefore for any $f$ convex 
\begin{align*}
\int_{\R^d}f(x)\, \d\mu_n(x)&= \int_{\R^d}\iint_{\R^d\times\R^d}f((1-t)x+ty)\, \d\eta^{n-1}_{z}\otimes\eta^{n-1}_{z}\d\mu(z)\\&\leq (1-t)\int_{\R^d}\int_{\R^d}f(x) \d\eta^{n-1}_{z}(x)\d\mu(z)+t\int_{\R^d}f(y) \d\eta^{n-1}_{z}(y)\d\mu(z)=\int_{\R^d}f(x)\,\d\mu_{n-1}(x).
\end{align*}
To show the convergence, let $\gamma_n\coloneqq \int_{\R^d}(\mathrm{Id},T_{z})_{\#}\eta_{z}^{n}\d \mu(z)$  with $T_{z}(x)\equiv z\, \forall\,  x \in \R^d$ be a plan. Then $\gamma_n\in \Gamma(\mu_n,\mu)$ since 
\begin{align*}
&(\pi^1)_{\#}\gamma_n(A)=\int_{\R^d}(\mathrm{Id},T_z)_{\#}\eta_z^n(A\times \R^d)\,\d\mu(z)=\int_{\R^d}\eta_z^n(A)\,\d\mu(z) = \mu_n(A),\quad \forall A\subseteq \R^{d}\\
&(\pi^2)_{\#}\gamma_n(A)=\int_{\R^d}(\mathrm{Id},T_z)_{\#}\eta_z^n(\R^d\times A)\,\d\mu(z)=\int_{\R^d}\chi_{A}(z)\d\mu(z)= \mu(A)
,\quad \forall A\subseteq \R^{d}. 
\end{align*}

%Moreover, by observing that 
%\begin{align*}
%M(\eta_z^n)=\int_{\R^d}x\, \d(\pi^t)_{\#}(\eta^{n-1}_{z}\otimes\eta^{n-1}_{z})(x)=\int_{\R^d\times\R^d}((1-t)x+ty)\, \d\eta^{n-1}_{z}(x)\otimes\eta^{n-1}_{z}(y)=z,  
%\end{align*}
%where we used the fact that thanks to Theorem \ref{thm:char_of_con_ord} for any $z$, $M(\eta_{z})=z$. We conclude again by using the again  Theorem \ref{thm:char_of_con_ord} in the if part: we observe that the family $\{\eta^{n}_{z}\}_{z}$ satisfies the hypotheses. 
So  
\begin{align*}
W_2^2(\mu_n,\mu)&\leq \int_{\R^d\times \R^d}|x-y|^2\d \gamma_n = \int_{\R^d}\int_{\R^d}|x-y|^2\d (\mathrm{Id},T_{z})_{\#}\eta_{z}^{n}\d \mu(z)\\&= \int_{\R^d}\int_{\R^d}|x-z|^2\d \eta_{z}^{n}(x)\d \mu(z)= \int_{\R^d}\int_{\R^d}|x-M(\eta_{z}^{n})|^2\d \eta_{z}^{n}(x)\d \mu(z)\\&= \int_{\R^d}\mathrm{Var}(\eta_{z}^{n})\d \mu(z)\leq (1-2t(1-t))^n\int_{\R^d}Var(\eta_{z})\d \mu(z),
\end{align*}
where the last inequality follows from Lemma \ref{lem:decreasing_variance}.
\end{proof}

\begin{proof}[Proof of Theorem \ref{thm:mon-conv}]
Let $\mu\preccurlyeq_C\nu$. Consider the sequence $\mu_n$ constructed as in Proposition~\ref{prop:mon-converg-sequence}. Then for any $n\in \mathbb{N}$, $\cG(\mu_{n})\leq\cG(\mu_{n-1})$, since by definition of total convexity we have 
\begin{align*}
\cG(\mu_{n})=& \cG\left(\pi^t_{\#}\left(\int_{\R^d}\eta^{n-1}_{z}\otimes\eta^{n-1}_{z}\d\mu(z)\right)\right)\\&\leq (1-t)\cG\left(\pi^1_{\#}\left(\int_{\R^d}\eta^{n-1}_{z}\otimes\eta^{n-1}_{z}\d\mu(z)\right)\right)+ t\cG\left(\pi^2_{\#}\left(\int_{\R^d}\eta^{n-1}_{z}\otimes\eta^{n-1}_{z}\d\mu(z)\right)\right)\\&=(1-t)\cG(\mu_{n-1})+t\cG(\mu_{n-1}) = \cG(\mu_{n-1}).
\end{align*}
 By lower semicontinuity of $\cG$, we conclude that 
 \begin{align*}
  \cG(\mu)\leq \liminf_{n\to+\infty} \cG(\mu_n)\leq \cG(\mu_0).
 \end{align*}
\end{proof}
\subsubsection{Non-expansivity of the $W_2$ backward projection}
As a consequence of the interplay between the notion of total convexity and the convex order of measures we deduce the $1$-Lipschitzianity of the Wasserstein projection on the set of lower bounds of an element, often called backward projection (see e.g. \cite{KiKiNa}). While this result has already been established in \cite{AlAu} and \cite{KiKiNa} via alternative techniques, we find it worthwhile to present it as a special case of the non-expansivity   of proximal operators within a broader class of functionals. The set of lower bounds of $\nu$ with respect to the order $\preccurlyeq_C$ is 
\begin{align*}
\cP_{\preccurlyeq_C\nu}\coloneqq\{\mu\in\cP(\R^d)\mid \mu\preccurlyeq_C\nu  \}.
\end{align*}
We denote by $\proj_{\preccurlyeq_C\nu}$ the $W_2$ projection on the set $\cP_{\preccurlyeq_C\nu}$, and with  $\mathbf{1}_{\cP_{\preccurlyeq_C\nu}}$ the indicator function of the same set: as we have already seen 
\begin{align}\label{eqn:proxcord}
\prox^{W}_{\tau\mathbf{1}_{\cP_{\preccurlyeq_C\nu}}}= \proj_{\cP_{\preccurlyeq_C\nu}}.
\end{align}
We start with the following Lemma.

\begin{lemma}\label{lem:Gnu}
For any $\nu\in \cP_2(\R^d)$, the set $\cP_{\preccurlyeq_C\nu}$ of lower bounds of $\nu$ for the convex order $\preccurlyeq_C$ is %closed for convex interpolating curves, meaning that if $\mu_0\preccurlyeq_C \nu$ and $\mu_1\preccurlyeq_C \nu$ then for any $t\mapsto \mu_t$ curve induced by a coupling, $\mu_t\preccurlyeq_C \nu$ for any $t\in [0,1]$.
totally convex. %In addition the  set $\cP_{\preccurlyeq_C\nu}$ is closed for the $W_2$ convergence.
Moreover given $\cG$ a proper, lower semicontinuous, totally $\lambda$-convex function whose domain intersects $\cP_{\preccurlyeq_C\nu}$ we have that
\begin{equation}\label{eqn:Gnu}
\cG^{\nu}(\mu):= \begin{cases} \cG(\mu) \quad & \text{ if }\mu \preccurlyeq_C \nu \\ +\infty & \text{ otherwise}\end{cases}
\end{equation}
%\mapsto \cG(\mu)\mathbf{1}_{\preccurlyeq_C\nu} (\mu)$
is proper, lower semicontinuous and totally $\lambda$-convex. 
%Therefore, the functional  $\nu\mapsto\mathbf{1}_{\preccurlyeq_C\nu}$ is proper, lower semicontinuous and totally convex.
\end{lemma}
\begin{proof}
Let $\mu_0,\mu_1 \in \cP_{\preccurlyeq_C\nu}$, i.e. $\mu_0\preccurlyeq_C\nu$,  $\mu_1\preccurlyeq_C\nu$.
Fix $\gamma\in \Gamma(\mu_0,\mu_1)$ and $t\in[0,1]$. For any $f$ convex 
\begin{equation*}
\int_{\R^d}f\,\d \mu_t=  \int_{\R^d \times \R^d}f((1-t)x+ty)\,\d \gamma(x,y)\leq(1-t)\int_{\R^d }f(x)\,\d\mu_0+t\int_{\R^d }f(y)\,\d\mu_1(y)\leq\int_{\R^d}f\,\d\nu, 
\end{equation*}
that is $\mu_t\preccurlyeq_C \nu$, thus $\mu_t \in \cP_{\preccurlyeq_C\nu} $, which proves that $\cP_{\preccurlyeq_C\nu}$  is totally convex.

We claim that $\cP_{\preccurlyeq_C\nu} $ is closed in the $W_2$ metric: notice in fact that in the definition of convex order we can restrict ourselves to use $f$ continuous and $1$-Lipschitz (see e.g. \cite{Goz17} at the beginning of Section~3). Since for these functions we have $\mu \mapsto \int f \mu$ is continuous in $W_2$, we can conclude that $\cP_{\preccurlyeq_C\nu} $ is closed. We deduce that $\mathbf{1}_{\cP_{\preccurlyeq_C\nu}}$ is proper, lower semicontinuous and totally convex.

For the last statement, we have by hypothesis that $\cG^{\nu}$ is proper. The other properties follow from the fact that $\cG^{\nu}= \cG + \mathbf{1}_{\cP_{\preccurlyeq_C\nu}}$.
%the total $\lambda$-convexity so it is non-empty, we can deduce that $\nu\mapsto\mathbf{1}_{\preccurlyeq_C\nu}$ is proper, lower semicontinuous and totally convex. 
\end{proof}

As a consequence of Lemma~\ref{lem:Gnu} we can prove that for the functionals of the form \eqref{eqn:Gnu} the proximal operator is Lipschitz and whenever $\cG$ is convex then it is non-expansive. As a particular case we recover that $\proj_{\preccurlyeq_C\nu}$ is $1$-Lipschitz.

\begin{theorem}
Let $\cG$ be a proper, lower semicontinuous, totally $\lambda$-convex functional and $\nu\in \cP_2(\R^d)$. Let $\cG^{\nu}$ be defined as in \eqref{eqn:Gnu}. Then for every $\tau>0$ such that $\lambda \tau>-1$ we have that $ \prox^{W}_{\tau \cG^{\nu}}$ is $\frac 1{1+\lambda \tau}$-Lipschitz. In particular 
\begin{itemize}
    \item[(i)] the Wasserstein projection on lower bounds  of $\nu$ with respect to the convex order $\preccurlyeq_C$, $\mu\mapsto \proj_{\cP_{\preccurlyeq_C\nu}}(\mu)$ is $1$-Lipschitz.
    \item[(ii)] the map $\mu \mapsto \argmin_{ \eta \preccurlyeq_C\nu} \{ \alpha Var(\eta) + W_2^2(\eta, \mu)\}$ is $\frac 1{1+\alpha}$-Lipschitz for every $\alpha >-1$.
\end{itemize}
\end{theorem}

\begin{proof} Thanks to Lemma~\ref{lem:Gnu} we have that $\cG^{\nu}$ is $\lambda$-convex. We can then apply Remark~\ref{rem:totlambda} to conclude that $ \prox^{W}_{\tau \cG^{\nu}}$ is $\frac 1{1+\lambda \tau}$-Lipschitz.
\begin{itemize}
    \item[(i)] Letting $\cG(\mu) \equiv 0$ we have $\cG^{\nu}=\mathbf{1}_{\cP_{\preccurlyeq_C\nu}}$ and so we conclude using~\eqref{eqn:proxcord}
    \item[(ii)] Let $b=M(\nu)$; we notice that if $\eta \preccurlyeq_C\nu$ then $M(\eta)=M(\nu)=b$. In particular
    \begin{equation}\label{eqn:fixed} Var(\eta)=M_2(\eta) - \|b\|^2 \qquad \forall \eta \; :\; \eta \preccurlyeq_C\nu.
    \end{equation}
    Let $\cG(\eta):=\alpha M_2(\eta)$; we have that $\cG$ is $2\alpha$-convex. Moreover, thanks to~\eqref{eqn:fixed},
    \begin{align*}
        \argmin_{ \eta \preccurlyeq_C\nu} \left\{ \alpha Var(\eta) + W_2^2(\eta, \mu) \right\} &= \argmin_{ \eta \preccurlyeq_C\nu} \left\{ \alpha M_2(\eta) + W_2^2(\eta, \mu) \right\} \\ &= \argmin_{ \eta \in \cP_2(\R^d)} \left\{ \cG^{\nu}(\eta) + W_2^2(\eta, \mu) \right\},
    \end{align*}
    which is by definition $\prox^{W}_{\frac 12 \cG^{\nu}}$: now we can conclude since $1+ \lambda \tau=1+\frac 12 \cdot 2\alpha=1+\alpha$.
\end{itemize}
\end{proof}
Notice that the quantitative stability of the \emph{forward} Wasserstein projection, that is, the projection onto the upper bounds of an element, has also been investigated in \cite{AlAu} and \cite{KiKiNa}; in particular, its $\frac{1}{2}$-H\"older continuity is established. 

%However, we cannot proceed as for the backward projection since, as shown by the following example, the set of upper bounds of an element, $\mathcal{P}_{\nu \preccurlyeq_C}\coloneqq\{\mu\mid \nu\preccurlyeq_C\mu\}$, is not closed under convex interpolating curves, meaning that the indicator function $\mathbf{1}_{\mathcal{P}_{\nu \preccurlyeq_C}}$ is not totally convex.

%\begin{example}
%The set of upper bounds of an element for the convex order  is not closed for convex interpolating curves. Indeed given $x,\, y\in \R^d$ and  $\nu= \frac{1}{2}\delta_x+\frac{1}{2}\delta_y$, consider $\gamma= \frac{1}{2}\delta_{(x,y)}+\frac{1}{2}\delta_{(y,x)}$ and the induced curve. Then $\nu\not\preccurlyeq_C\nu_{\frac 12}=(\pi_{\frac 12})_{\#}\gamma=\delta_{\frac{x+y}{2}}$.
%\end{example}

\subsubsection{On the domain and  set of minimizers of totally convex functionals}
\begin{theorem}\label{thm:deltamin}
Let $\cG: \cP_2(\R^d)\to (-\infty,+\infty]$ proper, lower semicontinuous and totally $\lambda$-convex. 
Then: 
\begin{itemize}
\item  $D(\cG)\cap \{\delta_x\mid x\in \R^d\}\neq \emptyset$;
\item if $\lambda=0$ and $\mu\in \argmin \cG$, then $\delta_{M(\mu)}\in \argmin \cG$.
\end{itemize}
%totally {\color{red} $\lambda$-}convex
 \end{theorem}
 \begin{proof}%[Proof of Theorem \ref{thm:deltamin}]
Follows by Proposition \ref{prop:delta-bar}, applied to $\cG- \frac{\lambda}2 M_2$, which, by \eqref{eqn:equiv_lambdaconv} is totally convex.
\end{proof}
In the following corollaries, this striking property lets us exclude functionals from being totally $\lambda$-convex just looking at their domains.
\begin{corollary} Let us consider $\cF: \P_2(\R^d)\to (-\infty, +\infty]$ be proper, weakly lower semicontinuous functional which is convex along generalized geodesics. Suppose that $D(\cF) \subseteq \P_2^{a.c.}(\R^d)$; then $\cF$ is not totally $\lambda$-convex for any $\lambda \in \R$.  
\end{corollary}
\begin{corollary}
The relative entropy $\Ent:\cP_2(\R^d)\to(-\infty,+\infty]$, 
\begin{align*}
\Ent(\mu)\coloneqq \begin{cases}
\int_{\R^d}\rho\log(\rho)\, \d \mathcal{L}^d &\text{if } \mu= \rho \mathcal{L}^d,\\
+\infty &\text{else}.
\end{cases}
\end{align*}is not totally $\lambda$-convex for any $\lambda \in \R$.  
\end{corollary}

\begin{comment}
\begin{lemma}\label{lem:delta-bar}
Let $\mu\in \cP_2(\R^d)$, then $\delta_{M(\mu)}$ a minimal element for $\preccurlyeq_C$ and its the minimum of the set $\preccurlyeq_C\mu$. 
\end{lemma}
\end{comment}
\begin{comment}
\begin{lemma}\label{lem:sequence-to-bar}
Let $\mu\in \cP_2(\R^d)$ and $t\in (0,1)$. Let $\{\mu_n\}_n\subseteq\cP_2(\R^d)$ be the sequence with $\mu_0\coloneqq\mu$ and $\mu_n$ recursively defined as 
\begin{align*}
\mu_{n+1}\coloneqq\pi^t_{\#}\gamma_n, \quad \text{with } \gamma_n\coloneqq\mu_n \otimes \mu_n, \qquad \text{for every } n\in \mathbb{N}.
\end{align*}
Then $\mu_n\to \delta_{M(\mu)}$ in $(\cP_2(\R^d),W_2)$.
\end{lemma}

\begin{proof}
If $\mu=\delta_x$ for some $x\in \R^d$, then  $\mu_n=\delta_x$ for any $n\in \mathbb{N}$ and there is nothing to prove. Otherwise, 
from Lemma \ref{lem:decreasing_variance} we have $M(\mu_n)=M(\mu)$ for every $n\in \N$ and $\Var(\mu_n)\to 0$ being $\Var(\mu_n)=(1-2(1-t)t)^n\Var(\mu)$. Since we have $W_2^2(\mu_n,\delta_{M(\mu)})=W_2^2(\mu_n,\delta_{M(\mu_n)})= \Var(\mu_n)$, we can conclude.
\end{proof}
\end{comment}

\subsubsection{Approximation with totally convex functionals}

The following Proposition rules out the possibility of studying the proximal operator of a functional which is convex along outer generalized geodesics through approximations with totally convex functionals (which we know having a non-expansive proximal operator). This is a consequence of the Lemma~\ref{lem:prox-delta} which we state later.

\begin{prop}\label{prop:prox_not_approx}
	Let $\cG:\cP_2(\R^d)\to (-\infty, +\infty]$ be proper functional, lower semicontinuous and $\lambda$-convex along outer generalized geodesics with $D(\cG)\subset \cP_2^{a.c.}(\R^d)$. Let $\tau>0$ such that $\lambda \tau>-1$. Then it cannot exist a sequence $(\cG_n)_n$ of proper, lower semicontinuous, totally $\lambda_n$-convex functionals with $\lambda_n\tau>-1$, such that $\prox_{\tau \cG_n}(\mu)\to \prox_{\tau \cG}(\mu)$ in $W_2$ (not even weakly) for every $\mu\in \cP_2(\R^d)$.
\end{prop}

\begin{proof}
	Suppose such sequence exists. We know from Lemma~\ref{lem:prox-delta} that for all $n$ it holds $\prox_{\tau \cG_n}(\delta_{x}) = \delta_{y_x^n}$ for some $x\in \R^d$ and sequence $\{y_x^n\}_n$. Since $\prox_{\tau \cG}(\delta_{x})\in \cP_2^{a.c.}(\R^d)$, we conclude.
\end{proof}

%% VECCHIA VERSIONE
 \begin{comment}
 
\begin{lemma}[Prox of deltas are deltas]\label{lem:prox-delta}

	Let $\cG:\cP_2(\R^d)\to (-\infty, +\infty]$ be proper, lower semicontinuous and totally convex. Fix  $\tau>0$.  For every $x\in \R^d$ there exists $y_x$ such that $\prox_{\tau \cG}(\delta_{x}) = \delta_{y_x}$.
\end{lemma}
\begin{proof}
Since $\cG$ is totally convex, then it is convex along outer generalized geodesics and thus $\prox_{\tau \cG}$ is single valued. Let $x\in \R^d$ and $\mu^*$ satisfying $\mu^* = \argmin_{\nu} \left\{ \cG(\nu)+\frac{1}{2\tau}W_2^2(\delta_x,\nu) \right\}$. Since it holds $\cG(\delta_{M(\mu^*)})\leq \cG(\mu^*)$ by Theorem \ref{thm:deltamin}, and moreover by Jensen inequality 
\[
W_2^2(\delta_x,\delta_{M(\mu^*)}) = \left\| x- \int_{\R^d} y \, d\mu^*\right\|^2\leq \int_{\R^d} \left\|x-y\right\|^2\, d\mu^*(y) = W_2^2(\delta_x,\mu^*),\]
therefore 
\begin{align*}
    \frac{1}{2\tau} W_2^2(\delta_x,\delta_{M(\mu^*)})+\cG(\delta_{M(\mu^*)})\leq \frac{1}{2\tau} W_2^2(\delta_x,\mu^*)+\cG(\mu^*), 
\end{align*}
%y using Theorem \ref{thm:deltamin}.
It follows that $\mu^*$ has to be a Dirac delta ($\mu^*=\delta_{M(\mu^*)}$). 
\end{proof}
\end{comment}

%%%VERSIONE $\lambda$-CONVEX

\begin{lemma}[Prox of deltas are deltas]\label{lem:prox-delta}

	Let $\cG:\cP_2(\R^d)\to (-\infty, +\infty]$ be proper, lower semicontinuous and totally $\lambda$-convex. Fix  $\tau>0$ such that $\lambda \tau > -1$.  For every $x\in \R^d$ there exists $y_x$ such that $\prox_{\tau \cG}(\delta_{x}) = \delta_{y_x}$.
\end{lemma}
\begin{proof}
Since $\cG$ is totally $\lambda$-convex, then it is $\lambda$-convex along outer generalized geodesics and thus $\prox_{\tau \cG}$ is single-valued. Let $\mu^*=\prox_{\tau \cG}(\delta_x)$.

Notice now that $ \mu \mapsto W_2^2(\delta_x, \mu)= \int \| x - y\|^2 \, d \mu (y) $ is a totally $2$-convex functional. In particular $\cF(\mu):=\frac 1{2\tau} W_2^2(\delta_x, \mu)+ \cG(\mu)$ is totally convex since $ \frac 1 {\tau}+\lambda  >0$. In particular by Theorem \ref{thm:deltamin} we have $\cF(\delta_{M(\mu^*)}) \leq \cF(\mu^*)$, therefore
\begin{align*}
    \frac{1}{2\tau} W_2^2(\delta_x,\delta_{M(\mu^*)})+\cG(\delta_{M(\mu^*)})\leq \frac{1}{2\tau} W_2^2(\delta_x,\mu^*)+\cG(\mu^*), 
\end{align*}
By the uniqueness of the minimizer it follows that $\mu^*$ has to be a Dirac delta ($\mu^*=\delta_{M(\mu^*)}$). \end{proof}

\subsection{Convex functionals with non convex Moreau envelope}\label{subsec:nonconvmoreau}

Let $\cG:\cP_2(\R^d)\to (-\infty,+\infty]$ be proper and lower semicontinuous.  Let $\cG^\tau$ be the Moreau envelope of $\cG$ defined by
\begin{equation}\label{eqn:moreau}
    \cG^\tau(\mu):= \inf_{\nu \in \cP_2(\R^d)} \cG(\nu)+\frac{1}{2\tau}W_2^2(\mu,\nu).
\end{equation}
We recall that this function is continuous, see \cite[Lemma 3.1.2]{AGS08}. While in Hilbert spaces the Moreau-Yosida regularization preserves the convexity of proper, lower semicontinuous functions, this property fails in the 2-Wasserstein space. However the unique example of the failure of this property we found in the literature is \cite[(3-1)]{CarlenCraig2013}, which is a highly degenerate functional, namely the indicator of a singleton $\cG=\mathbf{1}_{\{\mu\}}$. 

Here we find a very big class of functionals, which includes for example most internal energy functionals, for which this regularization property fails.

\begin{prop}
Let $d\geq 2$ and $\cG$ proper, lower semicontinuous and convex along generalized geodesics. Assume that $D(\cG)\cap \{\delta_x\mid x\in \R^d\}=\emptyset$. Then $\cG^{\tau}$ is not $\lambda$-convex along generalized geodesics for any $\tau$ such that $\lambda \tau >-1$. In particular for every $\tau>0$ we have that $\cG^{\tau}$ is not convex along generalized geodesics.
\end{prop}
\begin{proof}Let us prove the first point: arguing by contradiction suppose that $\cG^{\tau}$ is $\lambda$-convex for some $\lambda$ such that $\lambda+ \frac 1{\tau} >0$. Then there exists $\bar{\lambda} < \lambda$ with $\bar{\lambda}+ \frac 1{\tau} >0$. Let us consider the functional
     $$\cF^{\tau}(\mu)= \cG^{\tau}(\mu)- \frac{\bar{\lambda}}2 M_2(\mu):$$
    in particular we have $\cF^{\tau}$ is strongly convex along generalized geodesics and continuous, thus by \cite[Theorem 9.1]{CSS_2023} we have that it is totally strongly convex. In particular there exists a unique minimizer for $\cF^{\tau}$ and by Theorem~\ref{thm:deltamin} it is a Dirac delta: without loss of generality we can assume $\argmin \cF^{\tau} = \{\delta_0\}.$
    By definition of $\cG^{\tau}$ there exists a sequence $(\mu_n)_n \in \P_2(\R^d)$ such that $$\cG^{\tau}(\delta_0)=\lim_n\cG(\mu_n)+ \frac 1{2\tau} W_2^2(\mu_n, \delta_0)= \lim_n\cG(\mu_n)+ \frac 1{2\tau} M_2(\mu_n).$$
    Notice that we have $\liminf_{n} W_2(\mu_n,\delta_0) > 0 $ otherwise by the lower semicontinuity of $\cG$ we would get $\cG(\delta_0) < \infty$: in particular there exists $\delta>0$ such that $M_2(\mu_n) >\delta$ for all $n \in \N$. Since $\delta_0$ is the minimizer for $\cF^{\tau}$, we have
    $$\cG^{\tau}(\delta_0) = \cF^{\tau}(\delta_0) \leq \cF^{\tau} (\mu_n) = \cG^{\tau}(\mu_n) - \frac {\bar{\lambda}}2 M_2(\mu_n).$$
    In particular 
    $$\cG(\mu_n) + \frac 1{2 \tau} M_2(\mu_n) \geq  \cG^{\tau} (\delta_0) + \frac 12 \left( \bar{\lambda} + \frac 1{\tau} \right) M_2(\mu_n) \geq \cG^{\tau} (\delta_0) + \frac \delta 2 \left( \bar{\lambda} + \frac 1{\tau} \right);$$
    taking the limit as $n \to \infty$ we reach a contradiction.

    The second point is a direct consequence of the first point: in fact if $\lambda=0$ it is valid for every $\tau>0$.

\end{proof}

We underline that $d\geq 2$ is needed since if $d=1$ then if $\cG$ is convex along geodesics also $\cG^{\tau}$ is convex along geodesics since $(\P_2(\R), W_2)$ is isometric to a convex cone in a Hilbert space.

Moreover if $\cG$ is totally convex the Moreau-Yosida regularization is totally convex itself, as shown in \cite[Section 3.1]{PiSa}. We provide here a stronger converse if $d\geq 2$, namely that geodesic convexity of every regularization yields total convexity of the starting functional.

\begin{theorem}  Let $d \geq 2$ and $\cG$ proper, lower semicontinuous and convex along generalized geodesics. Assume that $\cG^{\tau}$ is convex along geodesics for any $\tau>0$. Then $\cG$ is totally convex.
\end{theorem}
\begin{proof}
We know that $\cG^{\tau}$ is $W_2$-continuous and convex along geodesics: in particular by \cite[Theorem 9.1]{CSS_2023} we have that $\cG^{\tau}$ is totally convex. Considering the validity of \eqref{eqn:totconv} for $\cG^{\tau}$ and using the fact that $\cG^{\tau} \to \cG$ pointwise for $\tau \to 0$, we can conclude that $\cG$ is totally convex as well.
\end{proof}

    %It would be interesting to explore the convexity conditions on $\cG$ under which the convexity of the Moreau-Yosida regularization can be recovered. 

\section{Non-expansivity and weak non-expansivity} \label{sec:weaklyconvex}

Here, we consider some special cases, based on conditions placed either on the starting points or on the functionals, where generalizations of  non-expansivity   condition holds.

\subsection{Convexity along geodesics and weak non-expansivity  }\label{subsec:weak-geo}
%\subsubsection*{Existing results}

In \cite{adve2020nonexpansiveness}, the authors introduce the notion of weak non-expansivity   to investigate the non-expansivity   of the Wasserstein projection operator on the set of absolutely continuous probability measures with density bounded above by a constant. 
In particular, denoting $K\coloneqq \{\mu=\rho\mathcal{L}^d\in \cP_2(\Omega)\mid 0\leq\rho\leq \lambda \}$  for $\Omega\subseteq \R^d$  convex and $\lambda>0$, they have the following.
%Recall that ${K} := \{\rho \in \proj^{a.c.}(\Omega) : \rho \le 1\}$ and  $P_{{K}}$  the $W_2$ projection onto ${K}$. Throughout this section, let $\mu, \nu \in \mathcal{P}_2(\Omega)$, and set $\tilde{\mu} := \proj_{{K}}(\mu)$ and $\tilde{\nu} := \proj_{{K}}(\nu)$ {\color{red}[E: not in our result]}. Denote the optimal transport plan from $\tilde{\mu}$ to $\tilde{\nu}$ by $\eta$. Note that since $\tilde{\mu}, \tilde{\nu}$ are absolutely continuous, $\eta$ is unique and induced by a map.
%In the theorem below one bounds the distance squared between $\mu$ and $\nu$ by the transportation cost of a slightly suboptimal transport plan. This is a sort of ``weak nonexpansivity.''

\begin{theorem}[Theorem 3.1 \cite{adve2020nonexpansiveness}]\label{thm:weak-admes}
Let $\Omega \subseteq \mathbb{R}^d$ be a closed convex set. Let $T, U : \Omega \to \Omega$ stand for the optimal maps from $\tilde \mu\coloneqq\proj_K({\mu}), \tilde \nu\coloneqq\proj_K({\nu})$ to $\mu, \nu$, respectively. Take  $\pi := (T, U)_{\#}\eta \in \Gamma(\mu, \nu)$ where $\eta\in \Gamma_{o}(\tilde \mu, \tilde \nu)$. Then
\[
W_2^2(\tilde \mu, \tilde \nu) \le \int_{\Omega \times \Omega} |x - y|^2 \, \mathrm{d}\pi(x,y).
\]
\end{theorem}
This property is called in \cite{adve2020nonexpansiveness} \emph{weak non-expansivity}. By recalling \eqref{eq:prox=proj} this is a weak non-expansivity  result for the proximal operator of the functional $ \mathbf{1}_K$.
We generalize the result to all convex functionals.\\ 

We first formalize the definition of weak non-expansivity. \begin{definition}
  An operator $A: \mathcal{P}_2(\mathbb{R}^d) \to \mathcal{P}_2(\mathbb{R}^d)$ is said to be weakly non-expansive if for every $\mu,\nu \in \mathcal{P}_2(\mathbb{R}^d)$ there exists a nonempty subset $\Gamma_A(\mu,\nu) \subseteq \Gamma(\mu, \nu)$ such that 
$$W^2_2(A(\mu), A(\nu)) \leq \inf \left\{ \int_{\mathbb{R}^d \times \mathbb{R}^d} \|x-y\|^2 \, d\gamma(x,y) \; : \; \gamma \in \Gamma_A(\mu, \nu)\right\}.$$  
\end{definition}

In this section we fix as coordinates of $(\R^d)^4$,  $(x, \tilde{x}, \tilde{y}, y)$, we denote by $\pi^{x}, \pi^{\tilde x}, \pi^{y}, \pi^{\tilde y}: (\R^d)^4\to \R^d$,  the projections on the component associated to the corresponding coordinate and   $\pi^{x,\tilde x}:(x,\tilde x, \tilde y, y)\mapsto (x,\tilde x)$ (analogously  for $\pi^{\tilde x,\tilde y}$ and $\pi^{y,\tilde y}$). Finally we denote by $\pi_t^{\tilde x\to \tilde y}$  the $t$ interpolation map $(x,\tilde x, \tilde y, y)\mapsto (1-t)\tilde x+t\tilde y$ of the variables associated to the coordinates $\tilde x$ and $\tilde y$ (and analogously for the other components). For clarity of exposition, we introduce the following definition. 

\begin{definition}\label{def:3geo}
   Given $\mu,\, \nu, \,\tilde{\mu},\, \tilde{\nu} \in \mathcal{P}_2(\mathbb{R}^d)$, a probability $\gamma \in \mathcal{P}((\mathbb{R}^d)^4)$ is a 3-geodesic plan with intermediate points $\tilde{\mu}$ and $\tilde{\nu}$ if $\gamma \in \Gamma(\mu, \tilde{\mu}, \tilde{\nu}, \nu)$ and its marginals satisfy $\pi_{\#}^{x,\tilde{x}}\gamma \in \Gamma_o(\mu, \tilde{\mu})$, $\pi_{\#}^{\tilde{x},\tilde{y}}\gamma \in \Gamma_o(\tilde{\mu}, \tilde{\nu})$, and $\pi_{\#}^{\tilde{y},y}\gamma \in \Gamma_o(\tilde{\nu}, \nu)$. %We call any curve $t\mapsto \pi^{t}_{\#}(\pi_\#^{x, y}\gamma)$ an interpolation  between $\mu$ and $\nu$ induced by a $3$-geodesic plan with intermediate points $\tilde \mu$ and $\tilde \nu$.  
    \end{definition}
This definition generalizes the plan obtained by gluing $( T, \mathrm{Id})_{\#} \proj_K(\mu)$, $\eta$, and $(\mathrm{Id}, U)_{\#} \proj_K(\nu)$,  in the notation of Theorem \ref{thm:weak-admes}. We use it to obtain an analogous weak non-expansivity  result in a more general setting in which one cannot assume the existence of maps.

\begin{theorem}[Weak non-expansivity   for geodesically convex functionals]\label{thm:weak_ne}
    Let $\cF:\P_2(\R^d)\to (-\infty,+\infty]$ be $\lambda$-convex along geodesics. Let $\mu,\nu \in \P_2(\R^d)$ and $\tilde \mu \in \prox_{\tau \cF}^W(\mu)$, $\tilde \nu \in \prox_{\tau \cF}^W(\nu)$ and let $\tau>0$ such that $\lambda \tau >-1$. Then 
    \[(1+\lambda \tau)^2W_2^2( \tilde{\mu}, \tilde{\nu}) \leq \inf_{\gamma} \left\{ \int_{(\R^d)^2} \|x-y\|^2 \, d \pi^{x,y}_{\#}\gamma - \int_{(\R^d)^4} \| x-y- (1+\lambda \tau) (\tilde{x}-\tilde{y})\|^2 \,d \gamma \right\}\]
    where the infimum is among all $3$-geodesic plans $\gamma\in \Gamma(\mu,\tilde\mu,\tilde \nu,\nu)$ with intermediate points $\tilde \mu$ and $\tilde \nu$ such that  $\cF$ is $\lambda$-convex along $t\mapsto (\pi_t^{\tilde x\to\tilde y})_{\#}\gamma$.
    In particular $(1+\lambda \tau)^2W_2^2(\tilde \mu, \tilde \nu)\leq \int \|x-y\|^2 \, d\eta$ for any $\eta =\pi^{x,y}_\#\gamma$ for some $\gamma \in \cP((\R^d)^4)$ as above. 
   \end{theorem}

   \begin{remark}
       Notice that for any $\mu$ and $\nu$ the set of plans in the infimum of the above theorem is always nonempty, since for any $\tilde \eta\in \Gamma_o(\tilde \mu, \tilde \nu)$, $\eta_\mu\in \Gamma_o(\mu, \tilde \mu)$ and $\eta_\nu\in \Gamma_o(\nu,\tilde \nu)$,  the Gluing Lemma ensures the existence of  a plan $\gamma$  on $(\R^d)^4$, such that
    \[\pi_\#^{x,\tilde x}\gamma=\eta_\mu  \quad \pi_\#^{y,\tilde y}\gamma=\eta_\nu, \quad \text{and} \quad \pi_\#^{\tilde x,\tilde y}\gamma=\tilde \eta.\] So, to satisfy the hypotheses of the theorem,  it is enough to take $\tilde \eta\in \Gamma_o(\tilde \mu, \tilde \nu) $ which induces a geodesic along which $\cF$ is $\lambda$-convex.
   \end{remark}
\begin{proof}
Let $
\gamma \in \Gamma(\mu,\tilde\mu,\tilde\nu,\nu)$,
be a $3$-geodesic plan with intermediate points $\tilde\mu$ and $\tilde\nu$ and such that  $\cF$ is $\lambda$-convex along $(\mu_t)_{t\in [0,1]}$ with $\mu_t\coloneqq(\pi_t^{\tilde x\to\tilde y})_{\#}\gamma$.
The following figure illustrates the situation. 

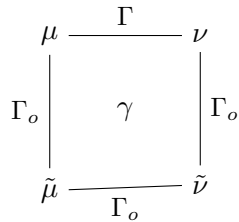
\begin{center}
\begin{tikzpicture}[
  node distance=1.5cm,
  every node/.style={font=\normalsize}
]
  % Corner measures (no boxes)
  \node (mu) {$\mu$};
  \node (nu) [right=of mu] {$\nu$};
  \node (tmu)  [below=of mu] {$\tilde \mu$};
  \node (tnu)  [below=of nu] {$\tilde \nu$};
 % Center gamma (exact center of the square)
  \node (gam) at (1,-1) {$\gamma$};
 % Square edges with labels (no dashed lines)
  \draw (mu) -- node[above] {\small{$  \Gamma$ }} (nu);
 \draw (tmu)  -- node[left]  {\small{$\Gamma_o$}} (mu);
  \draw (tnu)  -- node[right] {\small{$\Gamma_o$}} (nu);
 \draw (tmu)  -- node[below] {\small{$\Gamma_o$}} (tnu);
\end{tikzpicture}
\end{center}

We have
\begin{align}\label{eq:expanded cost}
    \int \|x-y\|^2 \, &\d\pi_\#^{x,y}\gamma(x,y) = \int \|\tilde x - \tilde y + x -\tilde x - (y-\tilde y)\|^2 \, d\gamma(x, \tilde x,\tilde y, y)\\
    \nonumber &= \int \|\tilde x-\tilde y\|^2 \, d\gamma + 2\int \langle \tilde x - \tilde y, x-\tilde x - (y -\tilde y)\rangle \, d\gamma   + \int \| x - \tilde x - (y-\tilde y)\|^2 \, d\gamma\\
    \nonumber &= W_2^2(\tilde \mu, \tilde \nu)+2\int \langle \tilde x - \tilde y, x-\tilde x - (y -\tilde y)\rangle \, d\gamma + \int \| x - \tilde x - (y-\tilde y)\|^2 \, d\gamma.
\end{align}
We need to prove a lower bound on $\int \langle \tilde x - \tilde y, x-\tilde x - (y -\tilde y)\rangle \, d\gamma(x, \tilde x,\tilde y, y)$. Since $\cF$ is $\lambda$-convex along $(\mu_t)_{t\in[0,1]}$, we have
\begin{align}\label{eq:proxmin}
\cF(\tilde \mu)+\frac{1}{2\tau} W_2^2(\tilde \mu, \mu)&\leq \cF(\mu_t)+\frac{1}{2\tau} W_2^2( \mu_t, \mu) \\& \leq (1-t) \cF(\tilde \mu) + t\cF(\tilde \nu) + \frac{1}{2\tau} W_2^2(\mu_t, \mu)- \frac{\lambda}{2}t(1-t)W_2^2(\tilde \mu, \tilde \nu), \notag
    \end{align}
\begin{align*}\cF(\tilde \nu)+\frac{1}{2\tau} W_2^2(\tilde \nu, \nu)& \leq \cF(\mu_{1-t})+\frac{1}{2\tau} W_2^2( \mu_{1-t}, \nu) \\ &\leq  t\cF(\tilde \mu) + (1-t) \cF(\tilde \nu) + \frac{1}{2\tau} W_2^2( \mu_{1-t}, \nu)-\frac{\lambda}{2}t(1-t)W_2^2(\tilde \mu, \tilde \nu),\end{align*}
which combined give
\[W_2^2(\tilde \mu, \mu) + W_2^2(\tilde \nu, \nu)+2\tau{\lambda}t(1-t)W_2^2(\tilde \mu, \tilde \nu) \leq W_2^2(\mu_t, \mu) + W_2^2( \mu_{1-t}, \nu).\]
On the other hand, for any $t\in [0,1]$, let $\pi^{x,(\tilde x\to \tilde y)_t}:(\R^d)^4\to(\R^d)^2$ be the map $(x,\tilde x, \tilde y, y)\mapsto (x,(1-t)\tilde x+t\tilde y)$, consider $\pi^{x,(\tilde x\to \tilde y)_t}_\# \gamma\in \Gamma(\mu, \mu_t)$.  Therefore, 
\begin{align} \nonumber 
W_2^2(\mu, \mu_t) & \leq \int \|x-z\|^2 \, d\pi^{x,(\tilde x,\tilde y)_t}_{\# }\gamma(x,z)= \int \|x- \tilde x + t (\tilde x - \tilde y)\|^2 \, d\gamma(x, \tilde x, \tilde y, y)\\
& \label{eq:nonopt} = \int \|x-\tilde x\|^2 \, d\gamma + 2 t \int \langle x-\tilde x, \tilde x - \tilde y\rangle \, d\gamma + t^2 \int \|\tilde x -\tilde y \|^2 \, d\gamma\\
& \nonumber= W_2^2(\mu,\tilde \mu) + 2 t \int \langle x-\tilde x, \tilde x - \tilde y\rangle \, d\gamma + t^2 W_2^2(\tilde \mu, \tilde \nu), 
\end{align}
where we used the fact that $\pi^{x,\tilde x}_\#\gamma \in \Gamma_{o}(\mu, \tilde \mu)$ and $\pi^{\tilde x,\tilde y}_\#\gamma \in \Gamma_{o}(\tilde\mu, \tilde \nu)$.
Similarly, $\pi^{(\tilde x,\tilde y)_{1-t},y}_\# \gamma\in \Gamma( \mu_{1-t},\nu)$, so that 
\[\begin{aligned}
W_2^2(\mu_{1-t},\nu) & \leq \int \|y-z\|^2 \, d\pi^{(\tilde x,\tilde y)_{1-t},y}_\# \gamma\\
&= W_2^2(\nu,\tilde \nu) + 2 t \int \langle y-\tilde y, \tilde y - \tilde x\rangle \, d\gamma + t^2 W_2^2(\tilde \mu, \tilde \nu), 
\end{aligned}\]
where we used the fact that $\pi^{\tilde y, y}_\#\gamma \in \Gamma_{o}(\tilde \nu,\nu)$ and $\pi^{\tilde x,\tilde y}_\#\gamma \in \Gamma_{o}(\tilde \mu, \tilde \nu)$.
Combining all the above and dividing by $t$, we obtain 
\[(2\lambda \tau(1-t)-2t )W_2^2(\tilde \mu, \tilde \nu) \leq 2\int \langle \tilde x - \tilde y, x-\tilde x - (y -\tilde y)\rangle \, d\gamma; \]
letting $t \to 0$ and rearranging we get
\[2\int \langle \tilde x - \tilde y, x- y \rangle \, d\gamma \geq 2(1+\lambda \tau) W_2^2(\tilde \mu, \tilde \nu).\]
Multiplying by $1+\lambda \tau$ and using that $2c \langle v,w\rangle= \|v\|^2 + c^2 \|w\|^2 - \|v-c w\|^2$ we finally get
\[\int \|x-y\|^2  \, d \gamma  + (1+\lambda \tau)^2 W_2^2(\tilde \mu, \tilde \nu) - \int \| x-y - (1+\lambda \tau)(\tilde{x} - \tilde{y}) \|^2 \, d \gamma \geq 2(1+\lambda \tau)^2 W_2^2(\tilde \mu, \tilde \nu),\]
thus we conclude.
\end{proof}

By this theorem one can show that the non-expansivity  property holds if one of the two starting measures is a delta.

\begin{corollary}
    Let $\cF:\P_2(\R^d) \to (-\infty, +\infty]$ be $\lambda$-convex along geodesics and let $\mu,\nu$ be such that $\Gamma(\mu,\nu)=\Gamma_o(\mu, \nu)$. Then 
    \[(1+\lambda \tau)^2W_2^2(\prox_{\tau \cF}^W(\mu),\prox_{\tau \cF}^W(\nu)) \leq W_2^2(\mu,\nu).\]
\end{corollary}
Notice that $\Gamma(\mu,\nu)=\Gamma_o(\mu,\nu)$ if for example $\Gamma(\mu,\nu)$ is a singleton. This happens if and only if $\mu$ or $\nu$ is a Dirac delta. Non-expansivity is recovered if $\lambda=0$.\\

It is worth noting that Theorem \ref{thm:weak_ne} can be seen as a  generalization to the Wasserstein space of the well-known property called  \emph{firm non-expansivity}  of proximal operators of convex functions in $\R^d$.  As the proof of the following corollary shows, for instance, if the two measures are Dirac deltas, one recovers  the  mentioned property. Although the result is largely known, we  provide the proof as a consequence of Theorem \ref{thm:weak_ne}. We recall that   \begin{align*}
\prox_{\tau V}(x)\coloneqq \argmin_{y\in \R^d}V(y)+\frac{1}{2\tau}||x-y||^2,
\end{align*}
is the classical proximal operator in $\R^d$. 

\begin{corollary}
    Let $V:\R^d\to (-\infty,+\infty]$ be proper, lower semicontinuous and convex and $x_0,y_0\in \R^d$. It holds
    \[\|\prox_{\tau V}(x_0)-\prox_{\tau V}(y_0)\|^2 \leq \|x_0-y_0\|^2 - \|x_0-\prox_{\tau V}(x_0)- (y_0-\prox_{\tau V}(y_0))\|^2.\]
   \end{corollary}
\begin{proof}
    Let $\cF: \mu \mapsto \int V \, d\mu$ and $\mu= \delta_{x_0}$, $\nu= \delta_{y_0}$ and set $p:= \prox_{\tau V}(x_0)$ and $q:=\prox_{\tau V}(y_0)$, then by recalling Lemma \ref{lem:prox-delta}
    \[\tilde \mu:=\prox_{\tau \cF}^W(\mu) = \delta_{p}, \qquad \tilde \nu := \prox_{\tau \cF}^W(\nu) = \delta_{q}.\]
    Apply Theorem~\ref{thm:weak_ne} to $\mu,\nu$ (with $\lambda=0$). The (unique) couplings
$\eta_\mu\in\Gamma_o(\mu,\tilde\mu)$, $\eta_\nu\in\Gamma_o(\nu,\tilde\nu)$, $\tilde\eta\in\Gamma_o(\tilde\mu,\tilde\nu)$
are Dirac measures:
\[
\eta_\mu=\delta_{(x_0,p)},\qquad
\eta_\nu=\delta_{(y_0,q)},\qquad
\tilde\eta=\delta_{(p,q)},
\]
hence the  plan obtained by gluing, $\gamma$, is necessarily $\gamma=\delta_{(x_0,p,q,y_0)}$ and $\pi^{x,y}_{\#}\gamma=\delta_{(x_0,y_0)}$.
Therefore,
\[
\int \|x-\tilde x-(y-\tilde y)\|^2\,d\gamma
=
\|x_0-p-(y_0-q)\|^2.
\]
Plugging into Theorem~\ref{thm:weak_ne} yields
\[
\|p-q\|^2
\le
\|x_0-y_0\|^2-\|x_0-p-(y_0-q)\|^2,
\]
which is the desired inequality.
\end{proof}

\subsection{Convexity along generalized geodesics and weak nonexpansivity}\label{subsec:weak-ggeo}

In what follows we observe that if the functional is also $\lambda$-convex along generalized geodesics, then the infimum in Theorem \ref{thm:weak_ne} can be done on a bigger set of plans, as explained in Remark \ref{rem:smaller-inf}.

\begin{theorem}\label{thm:gg_wne}
    Let $\cF:\P_2(\R^d)\to (-\infty,+\infty]$ be proper, lower semicontinuous, and $\lambda$-convex along generalized geodesics. Let $\mu,\nu \in \P_2(\R^d)$ and  $\lambda \tau > -1$. Then letting $\tilde{\mu}\coloneqq\prox_{\tau \cF}^W(\mu)$ and $\tilde{\nu}\coloneqq\prox_{\tau \cF}^W(\nu)$ we have  \[(1+\lambda \tau)^2W_2^2( \tilde{\mu}, \tilde{\nu}) \leq \inf_{\gamma} \left\{ \int \|x-y\|^2 \, d \pi^{x,y}_{\#}\gamma - \int \| x-y- (1+\lambda \tau) (\tilde{x}-\tilde{y})\|^2 \,d \gamma \right\}\]
    where the infimum is among all $3$-geodesic plans $\gamma\in \Gamma(\tilde{\mu},\mu,\tilde{\nu},\nu)$ with intermediate points $\mu$ and $\tilde{\nu}$ such that %by choosing as coordinates $(\tilde x, x, \tilde y, y)$, one has that 
    $\cF$ is $\lambda$-convex along  $t\mapsto (\pi^{\tilde x\to\tilde y}_t)_{\#}\gamma$, i.e. \eqref{eq:convgg} holds true.
\end{theorem}

\begin{remark}
The set in which we are taking the infimum is always nonempty. We can indeed take $\gamma\in \cP((\R^d)^3)$  such that, with coordinates $(x,\tilde x, \tilde y)$,  $t\mapsto (\pi^{\tilde x\to \tilde y}_t)_{\#}\gamma$ is a generalized geodesic between $\tilde \mu$ and $\tilde \nu$ with base $\mu$ and glue it with an $\eta_{\nu}\in\Gamma_o(\tilde \nu,\nu)$ through $\tilde \nu$.
\end{remark}
\begin{remark}
In Theorem \ref{thm:gg_wne} we can analogously take any $ \gamma\in \cP((\R^d)^4)$ such that $\pi_\#^{x,\tilde x} \gamma \in \Gamma_o(\mu,\tilde \mu)$  and $t\mapsto (\pi^{\tilde x\to\tilde y}_t)_{\#}\gamma$ is a generalized geodesic between $\tilde \mu$ and $\tilde \nu$ with base $\nu$, along which $\cF$ is convex.
\end{remark}
\begin{remark}\label{rem:smaller-inf}
 If $\cF:\P_2(\R^d)\to (-\infty,+\infty]$ satisfies the hypotheses of Theorem \ref{thm:gg_wne} given $\mu,\nu \in \P_2(\R^d)$ and $\tilde \mu\coloneqq\prox_{\tau \cF}^W(\mu)$, $\tilde \nu\coloneqq \prox_{\tau \cF}^W(\nu)$. Then as upper bound for  $W_2^2(\tilde \mu, \tilde \nu)$ we can take the infimum  among the upper bound given by Theorem \ref{thm:weak_ne} and the upper bound given by Theorem \ref{thm:gg_wne}. Notice moreover that this quantity can also be bounded by the so called $\tilde \nu$-based Wasserstein (semi-)distance \cite[Definition 3]{NennaPass2023, Tanguy2026}, also indicated as $W_{\tilde{\nu}}(\mu,\nu)$, which is strictly related to the quantity introduced in \cite[Equation (9.2.5)]{AGS08}.
\end{remark}

\begin{proof}
Let $\gamma$ be a $3$-geodesic plan $\gamma\in \Gamma(\tilde\mu,\mu,\tilde\nu,\nu)$ with intermediate points $\mu$ and $\tilde\nu$ such that by choosing as coordinates $(\tilde x, x, \tilde y, y)$, one has that $\cF$ satisfies the $\lambda$-convexity inequality \eqref{eq:convgg} along  $t\mapsto \mu_t\coloneqq(\pi^{\tilde x\to\tilde y}_t)_{\#}\gamma$.
We are in the following situation.
    
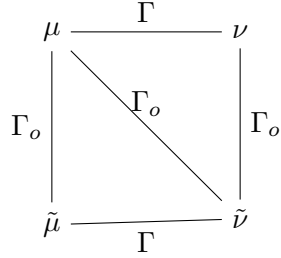
\begin{center}
\begin{tikzpicture}[
  node distance=2cm,
  every node/.style={font=\normalsize}
]
  % Corner measures (no boxes)
  \node (mu) {$\mu$};
  \node (nu) [right=of mu] {$\nu$};
  \node (tmu)  [below=of mu] {$\tilde \mu$};
  \node (tnu)  [below=of nu] {$\tilde \nu$};

  % Center gamma (exact center of the square)
  \node (bgam) at (0.75,-1.75) {};
  \node (bgam_p) at (1.75,-0.75) {};

  % Square edges with labels (no dashed lines)
  \draw (mu) -- node[above] {$\Gamma$} (nu);

  \draw (mu) -- node[above] {$\Gamma_o$} (tnu);

  \draw (tmu)  -- node[left]  {$\Gamma_o$} (mu);

  \draw (tnu)  -- node[right] {$\Gamma_o$} (nu);

  \draw (tmu)  -- node[below] {$ \Gamma$} (tnu);
\end{tikzpicture}
\end{center}
We proceed as in the proof of Theorem \ref{thm:weak_ne}. 
\begin{comment}
We have
\begin{equation}\label{eq:gg_estimate}
    \begin{aligned}
    \int \|x-y\|^2 \, d \pi_\#^{x,y}\gamma &= \int \|\tilde x - \tilde y + x -\tilde x - (y-\tilde y)\|^2 \, d\gamma\\
    & \hspace{-1cm} = \int \|\tilde x-\tilde y\|^2 \, d\gamma + 2\int \langle \tilde x - \tilde y, x-\tilde x - (y -\tilde y)\rangle \, d\gamma + \int \| x - \tilde x - (y-\tilde y)\|^2 \, d\gamma.
\end{aligned}\end{equation}
\end{comment}
By exploiting the convexity along  $t\mapsto \mu_t$ and the minimality of $\tilde\mu$ as in \eqref{eq:proxmin}, we have  
\begin{equation*}
    \begin{aligned}
       & \cF(\tilde\mu)+\frac{1}{2\tau} W_2^2(\tilde\mu, \mu)   \\
        & \leq (1-t) \cF(\tilde\mu) + t\cF(\tilde\nu) - \frac{\lambda}{2} t (1-t) \int \|\tilde x - \tilde y\|^2 \, d\gamma + \frac{1}{2\tau} W_2^2(\mu_t, \mu), 
         \end{aligned}
    \end{equation*}
and analogously
\begin{align*}
&\cF(\tilde\nu)+\frac{1}{2\tau} W_2^2(\tilde\nu, \nu)\\ &
\leq  t\cF(\tilde\mu) + (1-t) \cF(\tilde\nu) - \frac{\lambda}{2} (1-t) t \int \|\tilde x - \tilde y\|^2 \, d\gamma + \frac{1}{2\tau} W_2^2( \mu_{1-t}, \nu).
\end{align*}
Combining the two inequalities, we obtain
\[W_2^2(\tilde\mu, \mu) + W_2^2(\tilde\nu, \nu) + 2\lambda \tau t(1-t)\int \|\tilde x-\tilde y\|^2 \, d\gamma \leq W_2^2( \mu_t, \mu) + W_2^2(\mu_{1-t}, \nu).\]
 Therefore reasoning as in \eqref{eq:nonopt}, one has 
\begin{align*} W_2^2(\mu,\mu_t)  \leq 
 W_2^2(\mu,\tilde \mu) + 2 t \int \langle x-\tilde x, \tilde x - \tilde y\rangle \, d\gamma + t^2 \int \| \tilde x- \tilde y \|^2 \, d \gamma , \\
 W_2^2(\nu,\mu_{1-t}) \leq  W_2^2(\nu,\tilde \nu) + 2 t \int \langle y-\tilde y, \tilde y - \tilde x\rangle \, d\gamma + t^2 \int \| \tilde x- \tilde y \|^2 \, d \gamma . 
\end{align*}
Combining all the above and dividing by $t$, we obtain 
\[(1-t) 2\lambda \tau \int \|\tilde x -\tilde y\|^2 \, d\gamma -2t \int \| \tilde x- \tilde y \|^2 \, d \gamma \leq 2\int \langle \tilde x - \tilde y, x-\tilde x - (y -\tilde y)\rangle \, d\gamma, \]
and letting $t \to 0$ and rearranging we get
\[2\int \langle \tilde x - \tilde y, x- y \rangle \, d\gamma \geq 2(1+\lambda \tau) \int \|\tilde x -\tilde y\|^2 \, d\gamma.\]
Multiplying by $1+\lambda \tau$ and using that $2c < v,w>= \|v\|^2 + c^2 \|w\|^2 - \|v-c w\|^2$ we finally get
\[\int \|x-y\|^2  \, d \gamma  + (1+\lambda \tau)^2 \int \| \tilde{x} - \tilde{y}\|^2 \, d \gamma - \int \| x-y - (1+\lambda \tau)(\tilde{x} - \tilde{y}) \|^2 \, d \gamma \geq 2(1+\lambda \tau)^2 \int \|\tilde x -\tilde y\|^2 \, d\gamma,\]
thus we conclude.

\begin{comment}
So
\[\begin{aligned}
    \int \|x-y\|^2 \, d\pi_\#^{x,y }\gamma  & = \int \|\tilde x-\tilde y\|^2 \, d\gamma + 2\int \langle \tilde x - \tilde y, x-\tilde x - (y -\tilde y)\rangle \, d\gamma + \int \| x - \tilde x - (y-\tilde y)\|^2 \, d\gamma\\
    & \geq (1+2\lambda \tau) \int \|\tilde x-\tilde y\|^2 \, d\gamma + \int \| x - \tilde x - (y-\tilde y)\|^2 \, d\gamma\\
    & \geq (1+2\lambda \tau) W_2^2(\tilde \mu, \tilde \nu)+ \int \| x - \tilde x - (y-\tilde y)\|^2 \, d\gamma.
\end{aligned}\]
\end{comment}
\end{proof}
\begin{remark}
    Notice that in the proof we actually provide estimates on the quantity 
    \[(1+\lambda \tau)^2 \int \|\tilde x-\tilde y\|^2 \, d\gamma + \int \| x - y - (1+\lambda \tau) ( \tilde x -\tilde y)\|^2 \, d\gamma,\]
    and  in general  $\int \|\tilde x-\tilde y\|^2 \, d\gamma$ is greater than $ W_2^2(\tilde \mu, \tilde \nu)$.
\end{remark}
Notice that in the previous theorem, we do not use the fact that $\prox_{\tau \cF}^W(\mu)$ is a singleton. In particular, the same statement holds by taking any $\tilde\mu\in \prox_{\tau \cF}^W(\mu)$ and any $\tilde\nu\in \prox_{\tau \cF}^W(\nu)$.  It is worth observing that 
the fact that $\prox_{\tau \cF}^W(\nu)$ is a singleton can be seen as a consequence of the theorem. In particular given 
  $\cF:\P_2(\R^d)\to (-\infty,+\infty]$ $\lambda$-convex along generalized geodesics, $\mu\in \P_2(\R^d)$, $\lambda \tau > -1$, and   $ \tilde\mu, \tilde\nu \in \prox_{\tau \cF}^W(\mu)$, one can take $\gamma\in \Gamma(\tilde\mu,\mu,\tilde \nu,\mu)$ such that $\pi^{x,y}_{\#}\gamma=(Id,Id)_{\#}\mu$ and $\cF$ satisfies the $\lambda$-convexity inequality  along $(\mu_{t})_{t\in [0,1]}$ with $\mu_t\coloneqq (\pi^{\tilde x\to\tilde y}_t)_{\#}\gamma$. This gives 
\begin{align*}
(1+\lambda \tau)^2W_2^2( \tilde{\mu}, \tilde{\nu}) \leq  \int \|x-y\|^2 \, d \pi^{x,y}_{\#}\gamma =0
\end{align*}
which gives $\tilde \mu= \tilde \nu$. 

\subsection{Convexity along 2-base generalized geodesics and non-expansivity}\label{subsec:two base}
In this section, we introduce a new class of curves, called 2-base generalized geodesics. In a sense, this framework generalizes the notion of generalized geodesics by building a transport between $\mu$ and $\nu$ using two base measures, $\tilde\mu$ and $\tilde\nu$:  the standard definition of generalized geodesics,  can be recovered as a special case when the two bases coincide. {However, notice that one  standard interpretation of generalized geodesics is lost: one can no longer view the interpolation as a convex combination of optimal maps in the tangent space of a reference measure.}

 Furthermore, we introduce a corresponding notion of convexity along these curves. We then show that this notion of convexity is sufficient to ensure that the proximal operator is non-expansive.

By definition, the notion of convexity along 2-base generalized geodesics, sits in between the total convexity and the convexity along generalized geodesics. While in Example~\ref{ex:3geodesic-cvx} we prove that the classical internal energy functionals are not convex along 2-base generalized geodesics, we do not know whether it is a strictly looser notion than total convexity.

\begin{definition}[2-base generalized geodesics]\label{def:2base}
    Given $\mu, \nu\in \P_2(\R^d)$ we say that $(\mu_t)_{t\in[0,1]}$ is a \emph{2-base generalized geodesic} between $\mu$ and $\nu$ with bases $\tilde\mu,\tilde \nu$ if it is an interpolation between $\mu$ and $\nu$ induced by a 3-geodesic plan, that is there exists a 3-geodesic plan $\gamma\in \Gamma(\mu,\tilde \mu,\tilde \nu,\nu)$ with intermediate points $\tilde \mu$ and $\tilde \nu$ such that $\mu_t\coloneqq (\pi^{x\to y}_t)_{\#} \gamma$. 
 \end{definition}
\begin{comment}
\begin{definition}[Convexity along 2-base generalized geodesics]
    We say that $\cF:\P_2(\R^d) \to (-\infty,+\infty]$ is convex along 2-base generalized geodesics if for any $ \mu,\nu, \tilde \mu,\tilde \nu$ there exists a 2-base generalized geodesic $t\mapsto \mu_t$ between $\mu$ and $\nu$ with bases $\tilde\mu$ and $\tilde \nu$ such that $\cF$ is convex along $\mu_t$.
\end{definition}
\end{comment}
\begin{definition}[Convexity along 2-base generalized geodesics]
    We say that $\cG:\P_2(\R^d) \to (-\infty,+\infty]$ is $\lambda$-convex along 2-base generalized geodesics if for any $ \mu,\nu, \tilde \mu,\tilde \nu\in\cP_2(\R^d)$ there exists a 2-base generalized geodesic $(\mu_s)_{s\in[0,1]}$ between $\mu$ and $\nu$ with bases $\tilde\mu$ and $\tilde \nu$ along which 
   \begin{align*}
		\cG(\mu_s)\leq (1-s)\cG(\mu)+s\cG(\nu)  -\frac \lambda2 s(1-s) \int \| x- y\|^2 \,d \gamma \quad \text{for every } s \in [0,1],
	\end{align*}
where $\gamma$ is a $3$-geodesic plan associated to the 2-base generalized geodesic in the sense of Definition \ref{def:2base}. When $\lambda=0$ we  say that $\cG$ is  convex along  2-base generalized geodesics.
\end{definition}

\begin{remark}
  Functionals that are convex along 2-base generalized geodesics are also convex along generalized geodesics (it is enough to choose $\tilde \mu=\tilde \nu$ in the definition).
\end{remark}
\begin{remark}
For the sake of completeness, we observe that, analogously to Definition \ref{def:3geo}, given $\mu_0,\mu_k$ and $\mu_1,\dots, \mu_{k-1}$ in $\cP_2(\R^d)$, one can say that a plan $\gamma\in \cP((\R^d)^{k+1})$, $\gamma\in\Gamma(\mu_0,\mu_1,\dots,\mu_{k-1},\mu_k)$, is  a \emph{$k$-geodesic plan} with intermediate points $\mu_1,\dots, \mu_{k-1}$ if by denoting the components of $(\R^{d})^{k+1}$ as $(x_0,\dots, x_k)$, one has $\pi^{i,i+1}_{\#}\gamma\in \Gamma_{o}(\mu_i,\mu_{i+1})$ for any $i=0, \dots, k-1$. Analogously one can say that $(\mu_{t})_{t\in[0,1]}$ is a \emph{$(k-1)$-base generalized geodesics} between $\mu_0$ and $\mu_k$ with bases $\mu_1,\dots,\mu_{k-1}$ if $\mu_t\coloneqq (\pi^{0 \to k}_{t})_{\#}\gamma$ with $\gamma$ a $k$-geodesic plan with intermediate points $\mu_1,\dots,\mu_{k-1}$. The definition of convexity along $(k-1)$-base generalized geodesics can be given accordingly. We conclude the remark by noticing that for $k=2$ one recovers the classical definition of generalized geodesics and the related convexity notion, while the degenerate case $k=1$ is the standard geodesic convexity.
\end{remark}
\begin{theorem}[Non-expansivity  for functions convex along 2-base generalized geodesics]\label{thm:tbgg_ne}
    Let $\cF:\P_2(\R^d)\to (-\infty,+\infty]$ be proper, lower semicontinuous and $\lambda$-convex along 2-base generalized geodesics, let $\tau>0$ such that $\lambda \tau >-1$, $\mu,\nu \in \P_2(\R^d)$ and $\tilde \mu \in \prox_{\tau \cF}^W(\mu)$, $\tilde \nu \in \prox_{\tau \cF}^W(\nu)$. Then 
    \[(1+\lambda \tau) W_2(\tilde \mu, \tilde \nu) \leq W_2(\mu,\nu).\]
\end{theorem}

\begin{proof} We consider only the case $\lambda=0$, the general case being similar to the proofs of Theorem~\ref{thm:weak_ne} and Theorem~\ref{thm:gg_wne}. By convexity along 2-base generalized geodesics, we know there exists a plan $\gamma$  on $(\R^d)^4$, with coordinates $(x, y, \tilde x, \tilde y)$, such that
    \[\pi_\#^{(x,\tilde x)}\gamma=\eta_\mu  \quad \pi_\#^{(y,\tilde y)}\gamma=\eta_\nu, \quad \text{and} \quad \pi_\#^{(\tilde x,\tilde y)}\gamma=\tilde \eta,\]
    with $\tilde \eta\in \Gamma(\tilde \mu, \tilde \nu)$, $\eta_\mu\in \Gamma_o(\mu, \tilde \mu)$ and $\eta_\nu\in \Gamma_o(\nu, \tilde \nu)$ and $\eta:= \pi_\#^{(x,y)}\gamma \in \Gamma_{o}(\mu,\nu)$, and such that $\cF$ is convex along $\tilde \mu_t := (\pi^{\tilde x\to \tilde y}_t)_{\#}\gamma$. We are in the following situation.
    \begin{center}
\begin{tikzpicture}[
  node distance=1.5cm,
  every node/.style={font=\normalsize}
]
  % Corner measures (no boxes)
  \node (mu) {$\mu$};
  \node (nu) [right=of mu] {$\nu$};
  \node (tmu)  [below=of mu] {$\tilde \mu$};
  \node (tnu)  [below=of nu] {$\tilde \nu$};

  % Center gamma (exact center of the square)
  \node (gam) at (1,-1) {$\gamma$};

  % Square edges with labels (no dashed lines)
  \draw (mu) -- node[above] {$\eta \in \Gamma_o$} (nu);

  \draw (tmu)  -- node[left]  {$\eta_\mu \in \Gamma_o$} (mu);

  \draw (tnu)  -- node[right] {$\eta_\nu \in \Gamma_o$} (nu);

  \draw (tmu)  -- node[below] {$\tilde \eta \in \Gamma$} (tnu);
\end{tikzpicture}
\end{center}
We have
\[\begin{aligned}
    W_2^2(\mu,\nu)& =\int \|x-y\|^2 \, d\eta = \int \|\tilde x - \tilde y + x -\tilde x - (y-\tilde y)\|^2 \, d\gamma\\
    & = \int \|\tilde x-\tilde y\|^2 \, d\gamma + 2\int \langle \tilde x - \tilde y, x-\tilde x - (y -\tilde y)\rangle \, d\gamma + \int \| x - \tilde x - (y-\tilde y)\|^2 \, d\gamma\\
    & \geq W_2^2(\tilde \mu, \tilde \nu)+ 2\int \langle \tilde x - \tilde y, x-\tilde x - (y -\tilde y)\rangle \, d\gamma + \int \| x - \tilde x - (y-\tilde y)\|^2 \, d\gamma
\end{aligned}\]
It suffices to show that $\int \langle \tilde x - \tilde y, x-\tilde x - (y -\tilde y)\rangle \, d\gamma\geq 0$. By convexity along $\tilde \mu_t$
\begin{equation*}
    \cF(\tilde \mu)+\frac{1}{2\tau} W_2^2(\tilde \mu, \mu)\leq \cF(\tilde \mu_t)+\frac{1}{2\tau} W_2^2(\tilde \mu_t, \mu) \leq (1-t) \cF(\tilde \mu) + t\cF(\tilde \nu) + \frac{1}{2\tau} W_2^2(\tilde \mu_t, \mu).
    \end{equation*}
We also have
\begin{equation*}\cF(\tilde \nu)+\frac{1}{2\tau} W_2^2(\tilde \nu, \nu)\leq \cF(\tilde \mu_{(1-t)})+\frac{1}{2\tau} W_2^2(\tilde \mu_{(1-t)}, \nu) \leq  t\cF(\tilde \mu) + (1-t) \cF(\tilde \nu) + \frac{1}{2\tau} W_2^2(\tilde \mu_{(1-t)}, \nu).\end{equation*}
Combining the two inequalities, we obtain
\[W_2^2(\tilde \mu, \mu) + W_2^2(\tilde \nu, \nu) \leq W_2^2(\tilde \mu_t, \mu) + W_2^2(\tilde \mu_{(1-t)}, \nu).\]
On the other hand, let $\tilde \gamma := (\pi^{x}, \pi^{\tilde x}, \pi_t)_\# \gamma\in \P_2(\R^d\times \R^d \times \R^d)$ and notice that $\pi_\#^{(x,z_t)} \tilde \gamma\in \Gamma(\mu, \tilde \mu_t)$ needs not to be optimal. Therefore, 
\[\begin{aligned} W_2^2(\mu,\tilde \mu_t) & \leq \int \|x-z_t\|^2 \, d\tilde \gamma(x,\tilde x, z_t) = \int \|x- \tilde x + t (\tilde x - \tilde y)\|^2 \, d\gamma(x,y, \tilde x, \tilde y)\\
& = \int \|x-\tilde x\|^2 \, d\gamma + 2 t \int \langle x-\tilde x, \tilde x - \tilde y\rangle \, d\gamma + t^2 \int \|\tilde x -\tilde y \|^2 \, d\gamma\\
&= W_2^2(\mu,\tilde \mu) + 2 t \int \langle x-\tilde x, \tilde x - \tilde y\rangle \, d\gamma + t^2 W_2^2(\tilde \mu, \tilde \nu).
\end{aligned}\]
Similarly, we define $\tilde \gamma' := (\pi^{y}, \pi^{\tilde y}, \pi_{(1-t)})_\# \gamma\in \P_2(\R^d\times \R^d \times \R^d)$, and obtain
\[\begin{aligned} W_2^2(\nu,\tilde \mu_{(1-t)}) & \leq \int \|y-z_{(1-t)}\|^2 \, d\tilde \gamma'(y,\tilde y, z_{(1-t)}) = \int \|y- \tilde y + t(\tilde y - \tilde x)\|^2 \, d\gamma(x,y, \tilde x, \tilde y)\\
& = \int \|y-\tilde y\|^2 \, d\gamma + 2 t \int \langle y-\tilde y, \tilde y - \tilde x\rangle \, d\gamma + t^2 \int \|\tilde y -\tilde x \|^2 \, d\gamma\\
&= W_2^2(\nu,\tilde \nu) + 2 t \int \langle y-\tilde y, \tilde y - \tilde x\rangle \, d\gamma + t^2 W_2^2(\tilde \mu, \tilde \nu).
\end{aligned}\]
Combining all the above and dividing by $t$, we obtain 
\[-2t W_2^2(\tilde \mu, \tilde \nu) \leq 2\int \langle \tilde x - \tilde y, x-\tilde x - (y -\tilde y)\rangle \, d\gamma, \]
and letting $t \to 0$ we obtain the desired result.
\end{proof}

%{\color{blue}\paragraph{How large is this class of functions? Examples.}
%Our goal now is to show that
%\[\text{totally cvx} \subsetneq \text{cvx along 2-base generalized geodesics} \subsetneq \text{cvx along generalized geodesics}.\]}
%To prove the last strict inclusion we use the following example.
The notion of convexity along 2-base generalized geodesics is strictly stronger than convexity along generalized geodesics, as proven in the following Example.
\begin{example}\label{ex:3geodesic-cvx}
Let \(\mu_1=\mathcal N(0,I_2)\) on \(\mathbb R^2\), and define the linear maps $L_i:x\mapsto S_i x$, $i=1,\dots, 3$, with associated matrices
\[
S_1=
\begin{pmatrix}
1&-1\\
-1&2
\end{pmatrix},
\qquad
S_2=
\begin{pmatrix}
1&2\\
2&5
\end{pmatrix},
\qquad
S_3=
\begin{pmatrix}
10&-3\\
-3&1
\end{pmatrix}.
\]
Each \(S_i\) is symmetric positive definite, since $\det S_1=\det S_2=\det S_3=1$ and the upper-left entries are positive. Define
\[
\mu_2=(L_1)_\#\mu_1,
\qquad
\mu_3=(L_2)_\#\mu_2,
\qquad
\mu_4=(L_3)_\#\mu_3.
\]
Each linear map $L_i:x\mapsto S_i x$ is the gradient of the convex quadratic potential $\varphi_i(x)=\frac12 x^T S_i x$, then \(S_1\) is optimal from \(\mu_1\) to \(\mu_2\), \(S_2\) is optimal from \(\mu_2\) to \(\mu_3\), and \(S_3\) is optimal from \(\mu_3\) to \(\mu_4\). The composition of the  maps $T=L_3\circ L_2\circ L_1$ is associated to the matrix 
\[
S = S_3S_2S_1 = \begin{pmatrix} 
-1&6\\
0&-1
\end{pmatrix}.
\]
Now consider the interpolation
\[
T_t=(1-t)I+tT.
\]
At \(t=\frac12\), we obtain that $T_{1/2}$ is associated to 
\[
\frac12(I+S)
=
\frac12
\begin{pmatrix}
0&6\\
0&0
\end{pmatrix}
=
\begin{pmatrix}
0&3\\
0&0
\end{pmatrix},
\]
which has rank one. Consequently, $(T_{1/2})_\#\mu_1$
is supported on the line \(\mathbb R\times\{0\}\), and hence it is not absolutely continuous with respect to the Lebesgue measure on \(\mathbb R^2\). In particular we have
\[
\operatorname{Ent}\bigl((T_{1/2})_\#\mu_1\bigr)=+\infty.
\]
On the other hand, both \(\mu_1\) and \(\mu_4=T_\#\mu_1\) are nondegenerate Gaussian measures, hence they have finite entropy. Therefore the entropy is not convex along the curve $t\longmapsto (T_t)_\#\mu_1$. A similar example shows that every internal energy is not convex along $(T_t)_{\#} \mu$ for some $\mu \in \P_2(\R^d)$.
\end{example}

Although it remains unclear whether convexity along 2-base generalized geodesics is strictly weaker than the notion of total convexity, we make the following observations. One could define a notion of convexity along 3-base generalized geodesics.
Following this approach, we observe that any plan induced by a linear transformation $T$, provided its associated matrix has positive determinant and is not a negative scalar multiple of the identity,  can be decomposed into the composition of four positive definite linear maps, by virtue of \cite{Ba}. In particular, this implies that for a fixed plan $\pi = (I, T)_{\#}\mu$ with such a $T$, we can write $T = S_4 S_3 S_2 S_1$. By selecting $(S_1)_{\#}\mu$, $(S_2 S_1)_{\#}\mu$, and $(S_3 S_2 S_1)_{\#}\mu$ as bases, we can define the plan $\gamma \in \mathcal{P}((\mathbb{R}^d)^5)$ as $\gamma = (I, S_1, S_2 S_1, S_3 S_2 S_1, S_4 S_3 S_2 S_1)_{\#}\mu$. It immediately follows that the marginal projection satisfies $\pi^{1,5}_{\#}\gamma = (I, T)_{\#}\mu = \pi$. Since $\gamma$ is a 3-geodesic plan inducing the unique 3-base generalized geodesic associated with these bases, we conclude that any functional that is convex along 3-base generalized geodesics is also convex along couplings induced by such linear maps. By approximation we can conclude the same for plans which are of the type $(I, T)_{\#}\mu$, with $T$ linear map associated to matrix with positive determinant.  

 The same conclusion cannot be easily reached when dealing with 2-base generalized geodesics.  In  \cite{Ba}, the set of matrices which can be decomposed as the product of $n$ positive definite symmetric matrices is characterized. In particular one can see that while the set of matrices which can be decomposed as the product of $4$ positive definite symmetric matrices is dense in the set of matrices with positive determinant, this is not true for the set of matrices which can be decomposed as the product of $3$ positive definite symmetric matrices.

 For the general nonlinear case we notice that the differential of the composition of $k$ smooth optimal maps is the composition of $k$ positive definite matrices. We conjecture that the set of compositions of $k=4$ optimal maps is dense in the set of orientation preserving diffeomorphisms while for $k=3$ this is not true.
 
% \begin{remark}
% \textcolor{red} {congettura k=4 possiamo fare tutte le mappe che sono orientation preserving come congettura: inglobare questo nel remark}
% \end{remark}
\subsection{Conditional $(1+\delta)$-Lipschitz estimates}\label{subsec:almostlip}

In this section we find conditions on $\mu$ and $\nu$, possibly depending on $\cF$, under which we can have an inequality of the type

\begin{equation}\label{eqn:1delta-Lip}
(1+\lambda \tau) W_2(\prox_{\tau \cF}^W (\mu) , \prox_{\tau \cF}^W (\nu)) \leq (1+\delta )  W_2(\mu, \nu).\end{equation}
\begin{comment}
The results we find are quantitative relaxations of the cases where a very restrictive assumption on either $\mu$ or $\nu$ gives \eqref{eqn:1delta-Lip} with $\delta=0$. Relaxing this condition to $\delta>0$ but small allows us to find a larger set of $\mu$ and $\nu$ on which \eqref{eqn:1delta-Lip} holds true.
\end{comment}
 Under highly restrictive assumptions on either $\mu$ or $\nu$, \eqref{eqn:1delta-Lip} holds with $\delta=0$. We obtain a quantitative relaxation of this case. Relaxing this condition to a small $\delta>0$ allows us to establish \eqref{eqn:1delta-Lip} for a significantly larger set of measures.

We start with a Lemma which is a consequence of the weak non-expansivity.

\begin{lemma}\label{lem:w2mumu-estimate}
    Let $\cF:\P_2(\R^d) \to (-\infty, +\infty]$ be proper, lower semicontinuous and $\lambda$-convex along generalized geodesics, $\tau>0$ such that $\lambda \tau >-1$. Let $\mu, \nu \in \P_2(\R^d)$ and let $\tilde \mu := \prox_{\tau \cF}^W(\mu)$, $\tilde \nu := \prox_{\tau \cF}^W(\nu)$. Then 
    \[ (1+\lambda \tau) W_2(\tilde \mu, \tilde \nu)\leq  W_2(\mu,\nu) + 2\min \left\{W_2(\tilde \mu,\mu), W_2(\tilde \nu,\nu)\right\}.\]
    %In particular, if either $W_2(\tilde \mu, \mu) $ or $W_2(\tilde \nu, \nu) $ is less than $W_2(\mu,\nu)$, then
    %\[W_2(\tilde \mu, \tilde \nu) \leq 3 W_2(\mu,\nu).\]
\end{lemma}
\begin{proof}
    From Theorem~\ref{thm:gg_wne} and the triangle inequality we have that 
    \[ (1+\lambda \tau) W_2(\tilde \mu,\tilde \nu) \leq \|x-y\|_{L^2( \gamma)} \leq   \|x-\tilde y\|_{L^2(\gamma)} +  \|\tilde y -y\|_{L^2( \gamma)},\]
    for any $\gamma$ as in Theorem \ref{thm:gg_wne}.  In particular from $\pi_\#^{(x,\tilde y)} \gamma \in \Gamma_o(\mu,\tilde \nu)$ and $\pi_\#^{(y,\tilde y)} \gamma \in \Gamma_o(\nu,\tilde \nu)$,  one has
    \[ (1+\lambda \tau) W_2(\tilde \mu, \tilde \nu) \leq W_2(\mu, \tilde \nu) + W_2(\nu, \tilde \nu) \leq  W_2(\mu, \nu) + 2 W_2(\nu, \tilde \nu). \]
    By exchanging the roles of $\mu$ and $\nu$, we obtain the thesis.
\end{proof}

In the next Corollary we show that if one of the two distances $W_2(\tilde \mu, \mu)$ or $W_2(\tilde \nu, \nu)$ is smaller compared to the distance of the starting points $W_2(\mu,\nu)$, then we have $(1+\delta)$- Lipschitz regularity.

\begin{corollary}\label{cor:nualmostmini}
    Let $\cF:\P_2(\R^d) \to (-\infty, +\infty]$  be proper, lower semicontinuous and $\lambda$-convex along generalized geodesics, $\tau>0$ such that $\lambda \tau >-1$. Let $\mu, \nu \in \P_2(\R^d)$,  $\tilde \mu := \prox_{\tau \cF}^W(\mu)$, $\tilde \nu := \prox_{\tau \cF}^W(\nu)$ and $\eps>0$. If $W_2(\nu, \tilde{\nu}) \leq \eps^2 $, then 
    \[ (1+\lambda \tau) W_2(\tilde \mu, \tilde \nu)\leq (1+ \eps)W_2(\mu,\nu) \qquad \forall \, \mu \text{ s.t. } W_2(\mu, \nu) \geq 2\eps.\]
    The condition $W_2(\nu, \tilde{\nu}) \leq \eps^2 $ holds if ${\rm dist}(\nu, \argmin \cF) \leq \eps^2$.
    %In particular, if either $W_2(\tilde \mu, \mu) $ or $W_2(\tilde \nu, \nu) $ is less than $W_2(\mu,\nu)$, then
    %\[W_2(\tilde \mu, \tilde \nu) \leq 3 W_2(\mu,\nu).\]
\end{corollary}

\begin{proof} Thanks to Lemma~\ref{lem:w2mumu-estimate} we can estimate $(1+\lambda \tau) W_2(\tilde \mu, \tilde \nu) \leq W_2(\mu, \nu) + 2 \eps^2$. Moreover our assumption gives $2 \eps^2 \leq \eps \cdot W_2(\mu, \nu)$, thus we get the conclusion. 

Finally, for the last point let us use $\mu_* \in \argmin{\cF}$ as a competitor in the definition of $\prox_{\tau \cF}^W(\nu)$:

$$\frac{ W_2^2(\nu,\tilde {\nu})}{2\tau} + \cF(\tilde \nu) \leq  \frac{ W_2^2(\nu,\mu_*)}{2\tau} + \cF(\mu_*) \leq  \frac{ W_2^2(\nu,\mu_*)}{2\tau} + \cF(\tilde \nu).$$
In particular $W_2^2(\nu,\tilde {\nu}) \leq W_2^2(\nu,\mu_*)$; we get the conclusion taking the infimum with $\mu_* \in \argmin{\cF}$.
\end{proof}

For the next result, recall that in \cite{adve2020nonexpansiveness} it is shown that \eqref{eqn:1delta-Lip} holds with $\delta=0$ if one of the two starting measures is a Dirac delta and $\lambda=0$. This is equivalent to state: if $\Var(\nu)=0$ then
$$ W_2^2(\tilde{\mu}, \tilde{\nu}) \leq W_2^2(\mu,\nu) \qquad \forall\,  \mu \in \P_2(\R^d).$$

The next theorem can be seen as a quantitative version of that statement. We start by collecting standard estimates for the Wasserstein distance using variance and barycenters.

\begin{lemma}\label{lem:var_estimate} Let $\mu, \nu \in \P_2(\R^d)$. Then for every $\eta \in \Gamma(\mu,\nu)$ we have
\begin{align}\int \|x-y\|^2 \, d\eta &\leq \| M(\mu) - M(\nu)\|^2 +\Bigl( \sqrt{\Var(\mu)} +\sqrt{\Var(\nu)} \: \Bigr)^2;\label{eq:var_estimate_above} \\
\int \|x-y\|^2 \, d\eta &\geq \| M(\mu) - M(\nu)\|^2 + \Bigl( \sqrt{\Var(\mu)} -\sqrt{\Var(\nu)}\: \Bigr)^2;\label{eq:var_estimate_below}\\
\int \|x-y\|^2 \, d\eta &\leq W_2^2(\mu,\nu) + 4 \sqrt{ \Var(\mu)  \Var(\nu)}.\label{eq:var_estimate}
\end{align}
\end{lemma}

\begin{proof}

For every $\eta \in\Gamma(\mu,\nu)$ let $C_{\eta}:= \int \| x-y\|^2 d \eta -\| M(\mu)-M(\nu)\|^2 - \Var(\mu) - \Var (\nu)$. We have
    \[\begin{aligned}
        \int \|x-y\|^2 \, d\eta &= \int \|M(\mu)-M(\nu)\|^2 + \|x-M(\mu)-(y-M(\nu))\|^2  \\
        & \hspace{1.5cm}+ 2 \langle M(\mu)-M(\nu), x- M(\mu) - (y-M(\nu)) \rangle d\eta\\
        &= \|M(\mu)-M(\nu)\|^2 + \int\|x-M(\mu)-(y-M(\nu))\|^2 \, d \eta \\
        & = \|M(\mu)-M(\nu)\|^2 +\Var(\mu)+\Var(\nu) -2 \int \langle x-M(\mu), y-M(\nu)\rangle \, d \eta.
        \end{aligned}\]
In particular, using Cauchy-Schwarz inequality on the last term we get $$ -2\sqrt{\Var(\mu) \Var(\nu)} \leq  C_{\eta} \leq 2 \sqrt{\Var(\mu) \Var(\nu)} \qquad \forall \eta \in \Gamma(\mu,\nu).$$
Rearranging the terms we get \eqref{eq:var_estimate_above} and \eqref{eq:var_estimate_below}. Finally, subtracting \eqref{eq:var_estimate_below} for $\eta_o \in \Gamma_o(\mu, \nu)$ and \eqref{eq:var_estimate_above}  we get \eqref{eq:var_estimate}.

\end{proof}

\begin{theorem} Let $\cF: \P_2(\R^d) \to (-\infty, +\infty]$ be proper, lower semicontinuous, $\lambda$-convex along generalized geodesics, $\tau>0$ such that $\lambda \tau >-1$. Let $0 < \eps <1$: if $\sqrt{\Var(\nu)} \leq \eps^2$, then 
    $$(1+\lambda \tau)^2W_2^2(\prox_{\tau \cF}^W(\mu), \prox_{\tau \cF}^W(\nu)) \leq (1+4\eps)   W_2^2(\mu, \nu) \qquad \forall \, \mu \text{ : } W_2(\mu, \nu) \geq 4 \eps.$$
\end{theorem}
\begin{proof} We assume without loss of generality that  $\Var(\nu)= \eps^4$. Let us consider $\mu$ such that $W_2(\mu, \nu) \geq 4 \eps$ and let $\tilde{\mu}\coloneqq \prox_{\tau \cF}^W(\mu)$, $\tilde{\nu}\coloneqq\prox_{\tau \cF}^W(\nu)$. Thanks to weak non-expansivity   and Lemma~\ref{lem:var_estimate} we have
\begin{equation}\label{eqn:Var_contr} (1+\lambda \tau)^2 W_2^2(\tilde{\mu}, \tilde{\nu}) \leq W_2^2(\mu, \nu) + 4 \sqrt{ \Var(\mu) \Var(\nu)} \leq W_2^2(\mu, \nu) + 4 \eps^2  \sqrt{\Var(\mu)}.\end{equation}
Suppose by contradiction $(1+\lambda \tau)^2 W_2^2(\tilde{\mu}, \tilde{\nu}) > (1+4\eps)W_2^2(\mu,\nu)$; then substituting in \eqref{eqn:Var_contr} we have
\begin{equation}\label{eqn:W2Var_1} W_2^2(\mu,\nu) < \eps\sqrt{\Var(\mu)}.
\end{equation}
Using \eqref{eq:var_estimate_below} for an optimal plan leads to
\begin{equation}\label{eqn:W2Var_1} \|M(\mu)-M(\nu)\|^2+\Bigl( \sqrt{\Var(\mu)}- \eps^2 \Bigr)^2 < \eps\sqrt{\Var(\mu)},  
\end{equation}
which rearranging the terms becomes: 
\begin{equation}\label{eqn:W2Var_2} \|M(\mu)-M(\nu)\|^2+\Bigl( \sqrt{\Var(\mu)}- (\eps^2+\tfrac{\eps}2 )\Bigr)^2 < \tfrac{\eps^2}4 +  \eps^3.
\end{equation}
Using \eqref{eq:var_estimate_above} and the triangular inequality in $\R^{d+1}$ we get
\begin{align*}
W_2(\mu,\nu) &< \| (M(\mu),\sqrt{\Var (\mu) } ) - (M(\nu), -\eps^2) \|_{R^{d+1}} \\
& \leq \| (M(\mu),\sqrt{\Var (\mu) } ) - (M(\nu), \tfrac{\eps}2+\eps^2) \|_{R^{d+1}} + \| (M(\nu), \tfrac{\eps}2+\eps^2) - (M(\nu), -\eps^2) \|_{R^{d+1}} \\
& < \sqrt{ \tfrac{\eps^2}4 +  \eps^3} + 2\eps^2 +\tfrac{\eps}2 \leq 4 \eps,
\end{align*}
which is a contradiction.
\end{proof}

%So that, again, if $ \Var(\mu)+ \Var(\nu) \leq W_2^2(\mu,\nu)$, we have Lipschitzianity. {\color{blue}Notice that by the above result adapted to strong convexity, we could even have nonespansivity...!}

\subsection{Non-expansivity for dilations}\label{subsec:affine}

Here we prove that for a special choice of $\mu$ and $\nu$ we have non-expansivity for every functional which is  convex along generalized geodesics; the particular case is $\nu= T_{\#}\mu$ when $T$ is a dilation.

\begin{prop}[Translation and dilations]
Let $\cF:\P_2(\R^d)\to(-\infty, +\infty]$ be proper, lower semicontinuous and $\lambda$-convex along generalized geodesics, $\tau>0$ such that $\lambda \tau >-1$. Let $\mu\in \P_2(\R^d)$ and $\psi (x):= \frac{\alpha}{2} \|x\|^2 + \langle b, x\rangle $ where $\alpha >-1$ and $b\in \R^d$. Consider $\nu \coloneqq (I+\nabla \psi)_{\#}\mu$  and $\tilde \mu \coloneqq \prox_{\tau \cF}^W(\mu)$, $\tilde \nu := \prox_{\tau \cF}^W(\nu)$. Then 
\begin{equation}\label{eq:ne_affine} (1+\lambda \tau) W_2(\tilde \mu, \tilde \nu) \leq W_2(\mu,\nu).\end{equation}
\end{prop}

\begin{proof}
 We first prove the statement for  $\mu$ and $\nu$ absolutely continuous with respect to the Lebesgue measure of $\R^d$. We have the following situation.
        \begin{center}
\begin{tikzpicture}[
  node distance=2cm,
  every node/.style={font=\normalsize}
]
  % Corner measures (no boxes)
  \node (mu) {$\mu$};
  \node (nu) [right=of mu] {$\nu$};
  \node (tmu)  [below=of mu] {$\tilde \mu$};
  \node (tnu)  [below=of nu] {$\tilde \nu$};

  % Center gamma (exact center of the square)
  \node (gam) at (1.25,-1.25) {$\gamma$};

  % Square edges with labels (no dashed lines)
  \draw[->] (mu) -- node[above] {$S$} (nu);

  \draw[<-] (tmu)  -- node[left]  {$T_\mu$} (mu);

  \draw[<-] (tnu)  -- node[right] {$T_\nu$} (nu);

  \draw (tmu)  -- node[below] {\tiny $(T_\mu,T_\nu \circ S)_\#\mu$} (tnu);
\end{tikzpicture}
\end{center}
where the maps $S: x \mapsto x+\nabla \psi$, $T_{\mu}= I+\nabla \varphi_\mu$ and $T_{\nu}= I+\nabla \varphi_\nu$ are optimal (and thus $\varphi_\mu$ and $\varphi_\nu$ are convex functions). Notice that $T_\nu\circ S = I +\nabla \psi + \nabla \varphi_\nu\circ (I+\nabla \psi)$. Since $\nabla \varphi_\nu\circ (I+\nabla \psi)= \nabla h$ with  $h(x):= \frac{1}{1+\alpha}\varphi_\nu((1+\alpha)x +b)$, we have that $T_\nu \circ S$ is an optimal map. Define now the plan $\gamma := (I, S, T_{\mu}, T_{\nu}\circ S)_\#\mu\in \Gamma(\mu,\nu,\tilde \mu, \tilde \nu)$. By choosing as coordinates of $(\R^d)^4$, $(x,y, \tilde x,\tilde y)$, we have that $\pi^{(x,\tilde x)}_\#\gamma =(I,T_\mu)_\#\mu$ which is optimal, $\pi^{(x,y)}_\#\gamma =(I,S)_\#\mu$, which is optimal, and $\pi^{(y,\tilde y)}_\#\gamma =(S,T_\nu\circ S)_\#\mu = (I,T_\nu)_\# (S_\#\mu) =  (I,T_\nu)_\# \nu$, which is also optimal. Consider $(\mu_t)_{t\in[0,1]}$, with $\mu_t\coloneqq (\pi^{\tilde x\to \tilde y}_t)_{\#}  \gamma$.  Let $\bar \gamma := \pi^{(x,\tilde x, \tilde y)}_\# \gamma = (I, T_\mu,T_\nu \circ S)_\# \mu$. Notice that $\tilde \mu_t:=(\pi_t)_\#\bar \gamma$, with $\pi_t(x, \tilde x, \tilde y) = (1-t) \tilde x + t\tilde y$, is the unique generalized geodesic between $\tilde \mu$ and $\tilde \nu$ with base point $\mu$ (see \cite[Remark 9.2.3]{AGS08}). We can then use the convexity along generalized geodesics, to get convexity along $(\tilde \mu_t)_{t\in [0,1]}$, reason exactly as in the proof of Theorem \ref{thm:tbgg_ne} and conclude.

Finally, let $\mu,\, \nu\in \P_2(\R^d)$, not necessarily absolutely continuous with respect to the Lebesgue measure. Let $\{\mu_n\}_n\subset \P_2^{ac}(\R^d)$ and $\{\nu_n\}_n\subset \P_2^{ac}(\R^d)$ be two sequences converging in $W_2$ to $\mu$ and $\nu$ respectively, with $\nu_n\coloneqq (I+\nabla \psi)_{\#}\mu_n$. Then, since $\prox_{\tau \cF}^W$ is continuous we also have $\tilde \mu_n\coloneqq\prox_{\tau \cF}^W(\mu_n)$ and $\tilde \nu_n\coloneqq\prox_{\tau \cF}^W(\nu_n)$ converging  to $\tilde \mu$ and $\tilde \nu$, respectively. We  can pass to the limit in inequality \eqref{eq:ne_affine}, which  holds for every $\mu_n,\nu_n\in \P_2^{ac}(\R^d)$ thanks to the previous step and get the desired result.
\end{proof}

\section{Local H\"olderianity under convexity along generalized geodesics}\label{sec:loc hol}

In this section we prove a local $\frac 12$-H\"{o}lder continuity result to every proximal operator of a functional which is convex along (outer) generalized geodesics.

The first quantitative result for a proximal operator that we know of is found in \cite[Proposition 2.3.4, Remark 2.3.5]{Roudneff2011ModelisationMD}. It deals with the $W_2$-projection on the set of density constrained measures defined on a convex set.

\begin{prop} Let $K = \{ \rho \in \cP_2^{a.c.}(\Omega) : \rho=\bar \rho\mathcal{L}^d,\, \bar\rho \leq 1 \}$, with $\Omega\subseteq \R^d$ convex, be the set of density constrained measures. For every $\rho \in \P_2^{a.c.}(\Omega)$, there exists a unique Wasserstein projection $\proj_K \rho := \arg\min_{\nu \in K} W_2(\rho, \nu)$. 
Moreover, $\proj_K$ is continuous and satisfies a local $\frac{1}{2}$-Hölder continuity: for any $\rho^1, \rho^2$ in a neighborhood of $K$ (i.e., $W_2(\rho^i, K) \leq C$), we have
\[
W_2(\proj_K \rho^1, \proj_K \rho^2) \leq \sqrt{W_2^2(\rho^1, \rho^2) + 2C \cdot  W_2(\rho^1, \rho^2)}.
\]
\end{prop}

This result was later generalized for projections on $K_f = \{ \varrho \in \mathcal{P}_2(\Omega) : \varrho \leq f \}$, the set of probability measures with density bounded above by $f$. In fact in \cite[Corollary 5.3]{DePhilippis2016} the strong-weak continuity of $\proj_{K_f}$ is shown (that is, $W_2$ convergence of $\mu_n$ implies weak convergence of $\proj_{K_f}$), and also in this case the quantitative stability is proven only if $f=1$.
% \begin{prop}(\cite[Corollary 5.3]{DePhilippis2016}).
% For fixed $f$, the map $P_{K_f} : \mathcal{P}_2(\Omega) \to \mathcal{P}_2(\Omega)$ defined through
% \[
% P_{K_f}\mu := \operatorname{argmin}\{W_2^2(\varrho, \mu) : \varrho \in K_f\}
% \]
% is continuous in the following sense: if $\mu_n \to \mu$ for the $W_2$ distance, then $P_{K_f}[\mu_n] \to P_{K_f}[\mu]$ in the weak convergence.
% Moreover, in the case where $f = 1$ and $\Omega$ is a convex set, the projection is also locally $\frac{1}{2}$-Hölder continuous for $W_2$ on the whole $\mathcal{P}_2(\Omega)$.
% \end{prop}
The authors note in a subsequent remark, that is an open question whether the projection $\proj_{K_1}$ is $1$-Lipschitz. Furthermore, when $f$ is not constantly equal to $1$, even establishing the continuity of the projection with respect to the Wasserstein distance seems to be a delicate issue.
%\subsubsection*{Our  results}
We want to remark that recently, in \cite{DeFa} a $\frac 12$-H\"{o}lder continuity result is proven for the problem of the regularized moment measures: even though the setting is different the techniques used to prove that result are similar to ours. We believe that this kind of H\"{o}lder stability result can be generalized to more general variational problems involving the Wasserstein distance.

Before stating the main theorem of this section, we state and prove the following lemma in the particular case of $(\P_2(\R^d),W_2)$.  We underline that it holds true in any metric space $(X,d)$ where the quadratic lower bound is in terms of the distance $d$ from a point. For a set $A\subseteq \P_2(\R^d) $ we define $\cF(A)=\inf_A \cF$, and for $\mu\in \P_2(\R^d)$ we define $\dist(\mu,A)\coloneqq\inf\{W_2(\mu,\nu)\mid\nu\in A\}$ as the classical distance from a set.

\begin{lemma}\label{lem:bound}
Let $\cF:\P_2(\R^d)\to (-\infty, +\infty]$ be a functional and $\tau>0$ be fixed. %Assume that $\cF$ is lower bounded by a linear function of the distance from a point, in the sense that there exists two constants $A,\, B$ in $\R$ such that $\cF(\mu)\geq A+B \cdot W_2(\mu,\bar \mu)$, for a fixed $\bar \mu\in D(F)$, for any $\mu\in\P_2(\R^d)$.
Assume that there exists $\mu_0 \in D(\cF)$, and $A,\tau_* \in \R$ such that $\tau_* > \tau$ and 
\begin{equation}\label{eqn:quadraticbound}\cF(\mu)\geq A-\frac{1}{2\tau_*}  W_2^2(\mu,\mu_0) \qquad \forall \mu\in\P_2(\R^d).
\end{equation}
%Then there exists $C(\tau, \bar \mu, r, A,B)>0$ such that for any $\mu\in B_r(\bar \mu)$,
Then, given $r>0$ and $\mu\in B_r(\mu_0)$, for every $\tilde{\mu} \in \prox_{\tau \cF}^W(\mu)$ we have
\begin{equation}\label{eqn:estimatertau}
W_2^2(\mu,\tilde{\mu})\leq \frac{2 \tau^* (\tau+\tau_*)}{(\tau_*-\tau)^2}r^2 + \frac{4\tau \tau_*}{\tau_*-\tau} ( \cF (\mu_0) -A).
\end{equation}
In particular the function $C^{\tau}_{\mu}:=\sup\{ W_2(\mu, \tilde \mu) : \tilde \mu \in \prox_{\tau \cF}^W (\mu)\}$ is locally bounded.
%exists a locally bounded function $\mu\mapsto C^{\tau}_{\mu}$ such that $W_2^2(\mu,\prox^{W}_{\tau \cF}(\mu))\leq C^{\tau}_{\mu}$.
Moreover 

\begin{itemize}
    \item[(1)] if $\cF$ is lower-bounded and $\arg \min \cF \neq \emptyset$ then $C^{\tau}_{\mu}\leq \dist(\mu,\arg\min \cF)$;
    \item[(2)] if $\mu \in D(\cF)$ then $C_{\mu}^{\tau} \leq \sqrt{2\tau}  \sqrt{\cF(\mu)- \cF(\prox_{\tau \cF}^{W}(\mu))}$.
\end{itemize}
\end{lemma}

\begin{proof} First, we notice that \eqref{eqn:quadraticbound} is equivalent to $\cF^{\tau}(\mu_0) \geq A$, where $\cF^{\tau}$ is the Moreau envelope defined in \eqref{eqn:moreau}. We can then apply \cite[Lemma~2.2.1]{AGS08} and   use $v=\tilde{\mu}$ in equation (2.2.4) of \cite{AGS08} to get a bound on $W_2^2(\mu, \tilde{\mu})$: we can finally obtain \eqref{eqn:estimatertau} by exploiting the fact that $\Phi(\mu,\tau; \tilde{\mu}) \leq \Phi(\mu,\tau;\mu_0)$, where, in the notations of \cite[Chapter 2]{AGS08}, $\Phi(\mu,\tau; \tilde{\mu})\coloneqq \frac{1}{2\tau}W_2^{2}(\tilde\mu,\mu)+\cF(\tilde \mu)$. 
Now we show the other two bounds.
\begin{itemize}
    \item[(1)] Assume $\arg\min \cF\neq \emptyset$ and let $\mu_1 \in \arg\min \cF$. By the fact that $\tilde{\mu} \in \prox_{\tau \cF}^W(\mu)$, we have
$$\cF(\tilde{\mu}) + \frac{ W_2^2(\mu, \tilde{\mu})}{2 \tau}  \leq \cF(\mu_1) + \frac{ W_2^2(\mu, \mu_1)}{2 \tau}.$$
Using now that $\cF(\tilde{\mu}) \geq \cF(\mu_1)$, we conclude that $W_2^2(\mu, \tilde{\mu}) \leq W_2^2(\mu,\mu_1)$; optimizing in $\mu_1$ and $\tilde{\mu}$ we get $C^{\tau}_{\mu} \leq \dist(\mu,\arg\min \cF)$.

\item[(2)]If $\mu\in D(\cF)$, from the definition of $\prox_{\tau \cF}^W(\mu)$, by taking $\mu$ as a competitor it holds
$$\cF(\tilde{\mu})+\frac{1}{2\tau}W^{2}_2(\mu,\tilde{\mu})\leq \cF(\mu),
$$
which gives $C_{\mu}^{\tau} \leq \sqrt{2\tau} \cdot \sqrt{\cF(\mu)- \cF(\prox_{\tau \cF}^{W}(\mu))}$.
\end{itemize}

\end{proof}

\begin{theorem}\label{thm:hold-cont}
    Let $\cF:\P_2(\R^d)\to (-\infty, +\infty]$ be proper, lower semicontinuous and $\lambda$-convex along outer generalized geodesics, $\tau>0$ such that $\lambda \tau > -1$. Then there exists a locally bounded function $\mu\mapsto C^{\tau}_{\mu}$ such that for every $\mu_1, \mu_2 \in \P_2(\R^d)$
\begin{equation}\label{eqn:cbounded}
(1+\lambda \tau) W_2^2(\prox^{W}_{\tau \cF}( \mu_1) , \prox^{W}_{\tau \cF}(\mu_2) )\leq W_2^2(\mu_1,\mu_2) +  (C^{\tau}_{\mu_1}+C^{\tau}_{\mu_2}) W_2(\mu_1, \mu_2)%\quad \forall\,  \mu_1,\mu_2 \in \P_2(\R^d);
\end{equation}
in particular $\prox_{\tau \cF}^W$ is locally $\frac 12$-H\"{o}lder continuous. Moreover
\begin{itemize}
    \item[(i)] if $\argmin (\cF) \neq \emptyset$, then $C^{\tau}_{\mu} \leq \dist(\mu, \argmin{\cF}) $;
    \item[(ii)] if $\mu \in D(\cF)$ then $C^{\tau}_{\mu} \leq \sqrt{2\tau} \sqrt{\cF(\mu)- \cF(\prox_{\tau \cF}^{W}(\mu))}$;
    \item[(iii)] if $\mu \in D(\nabla^-\cF)$ then $C^{\tau}_{\mu} \leq \frac{\tau}{1+\lambda \tau} | \nabla^- \cF|(\mu)$, 
\end{itemize} 
where $|\nabla^{-}\cF|$ is the metric slope of $\cF$ defined as $|\nabla\cF^-|(\mu)\coloneqq\limsup_{\nu\to \mu}\frac{(\cF(\mu)-\cF(\nu))^+}{W_2(\mu,\nu)}$.
\end{theorem}
%\sara{Che cos'è $\nabla^-\cF$? $| \nabla^- \cF|(\mu)$ è quella che Carlen e Craig chiamano metric slope?}
%\simone{si, in AGS si chiama descending slope}

We recall that $\mu\mapsto \frac{W_2^2(\mu, \bar \mu)}{2}$ is $1$-convex along generalized geodesics with base point $\bar \mu$.
\begin{proof} For the first part we consider $\lambda=0$, the general case being quite similar. Let 
$$\Phi_{\mu}(\nu)\coloneqq  \cF(\nu) + \frac{1}{2\tau}W_2^2(\mu,\nu).$$
Let $(\rho^1_t)_{t\in [0,1]}$ be the generalized geodesic with base point $\mu_1$ from $\prox^{W}_{\tau \cF}(\mu_1) $ to $\prox^{W}_{\tau \cF}(\mu_2) $ and $(\rho^2_t)_{t\in [0,1]}$ the generalized geodesic with base point $\mu_2$ from $\prox^{W}_{\tau \cF}(\mu_2) $ to $\prox^{W}_{\tau \cF}(\mu_1)$. We know that $\nu \mapsto \Phi_{\mu_1}(\nu)$ and $\nu \mapsto \Phi_{\mu_2}(\nu)$ are $\frac{1}{\tau}$-convex respectively along $(\rho^1_t)_{t\in [0,1]}$ and $(\rho^2_t)_{t\in [0,1]}$. Thus for any $t\in [0,1]$ we have
 \begin{align*}
 \Phi_{\mu_1}(\rho^1_t)\leq(1-t) \Phi_{\mu_1}(\prox^{W}_{\tau \cF}(\mu_1))+t\Phi_{\mu_1}(\prox^{W}_{\tau \cF}(\mu_2))-\frac{t(1-t)}{2\tau}W_2^2(\prox^{W}_{\tau \cF}(\mu_1),\prox^{W}_{\tau \cF}(\mu_2)),\\
  \Phi_{\mu_2}(\rho^2_t)\leq(1-t) \Phi_{\mu_2}(\prox^{W}_{\tau \cF}(\mu_2))+t\Phi_{\mu_2}(\prox^{W}_{\tau \cF}(\mu_1))-\frac{t(1-t)}{2\tau}W_2^2(\prox^{W}_{\tau \cF}(\mu_1),\prox^{W}_{\tau \cF}(\mu_2)), 
 \end{align*} 
 that, when added, give 
 \begin{align*}
 \Phi_{\mu_1}(\rho^1_t)&+\Phi_{\mu_2}(\rho^2_t)\leq (1-t)\left( \Phi_{\mu_1}(\prox^{W}_{\tau \cF}(\mu_1))+\Phi_{\mu_2}(\prox^{W}_{\tau \cF}(\mu_2))\right)\\& +t\left(\Phi_{\mu_1}(\prox^{W}_{\tau \cF}(\mu_2))+\Phi_{\mu_2}(\prox^{W}_{\tau \cF}(\mu_1))\right)-\frac{t(1-t)}{\tau}W_2^2(\prox^{W}_{\tau \cF}(\mu_1),\prox^{W}_{\tau \cF}(\mu_2)).
 \end{align*}
 By using that $\prox^{W}_{\tau \cF}(\mu_i)$ is minimizer of $\Phi_{\mu_i}$ for $i=1,2$, we have
 \begin{align*}
& \frac{t(1-t)}{\tau}W_2^2(\prox^{W}_{\tau \cF}(\mu_1),\prox^{W}_{\tau \cF}(\mu_2))\\
 &\leq t(\Phi_{\mu_2}(\prox^{W}_{\tau \cF}(\mu_1))-\Phi_{\mu_1}(\prox^{W}_{\tau \cF}(\mu_1)))+t(\Phi_{\mu_1}(\prox^{W}_{\tau \cF}(\mu_2))-\Phi_{\mu_2}(\prox^{W}_{\tau \cF}(\mu_2))).
 \end{align*}
Dividing by $t$ and sending $t\to  0$,
 \begin{align*}
 &W_2^2(\prox^{W}_{\tau \cF}(\mu_1),\prox^{W}_{\tau \cF}(\mu_2))\\&\leq \frac{1}{2}(W_2^2(\prox^{W}_{\tau \cF}(\mu_1),\mu_2)-W_2^2(\prox^{W}_{\tau \cF}(\mu_1),\mu_1)+W_2^2(\prox^{W}_{\tau \cF}(\mu_2), \mu_1)-W_2^2(\prox^{W}_{\tau \cF}(\mu_2), \mu_2))\\ &\leq \frac{W_2(\mu_1,\mu_2)}{2} \left(W_2(\prox^{W}_{\tau \cF}(\mu_1),\mu_2)+W_2(\prox^{W}_{\tau \cF}(\mu_1),\mu_1)+W_2(\prox^{W}_{\tau \cF}(\mu_2), \mu_1)+W_2(\prox^{W}_{\tau \cF}(\mu_2), \mu_2)\right)\\ &\leq {W_2(\mu_1,\mu_2)}\left(W_2(\mu_1,\mu_2)+W_2(\prox^{W}_{\tau \cF}(\mu_1),\mu_1)+W_2(\prox^{W}_{\tau \cF}(\mu_2), \mu_2)\right). 
 \end{align*}

In particular we have that \eqref{eqn:cbounded} holds true for $C_{\mu}^{\tau}:=W_2(\mu,\prox_{\tau \cF}^W(\mu))$; we can now apply Lemma \ref{lem:bound} since geodesically  $\lambda$-convex functionals in $\P_2(\R^d)$ are lower bounded by a quadratic function of the distance e.g. %\cite[Theorem 4.3]{NaSa} and
\cite[Lemma 2.17]{MuSa} or \cite[Lemma~2.4.8]{AGS08} (in particular \eqref{eqn:quadraticbound} holds true for every $\tau_*$ satisfying ${\lambda}\tau_* >- 1$, and so we can choose $\tau_*>\tau$). We conclude that $C_{\mu}^\tau$ is locally bounded. Thus the first part of the statement, and (i) and (ii) follow respectively from Lemma~\ref{lem:bound} (1) and (2). The estimate in (iii) is instead proven, for example, in \cite[Lemma~3.1.6~(i)]{AGS08}.
% \begin{itemize}
%     \item[(iii)] We will prove $ W_2(\mu, \prox_{\tau \cF}^W(\mu)) \leq\frac{ \tau}{1+\lambda \tau} |\nabla^- \cF|(\mu)$.  It even though it is quite classical for a $\lambda$-geodesically convex functional we couldn't find a reference. Let us define $D:= W_2(\mu, \prox_{\tau \cF}^W(\mu))$ and consider the geodesic $\mu_t$ between $\mu$ and $\prox_{\tau \cF}^W(\mu)$ along which $\cF$ is $\lambda$-convex. Let $f(t)=\cF(\mu_t)$; we know that $g(t) = f(t) + \frac {t^2D^2}{2\tau}$ attains its minimum in $t=1$ thus $g'_-(1)=f'_-(1)+ \frac{D^2}{\tau} \leq 0$. In particular $D^2 \leq - \tau f'_-(1)$; by the $\lambda D^2$-convexity of $f$ we have $f'_+(0) \leq  f'_-(1) - \lambda D^2$ and by definition of descending slope
% $$ |\nabla^- \cF|(\mu)\geq \limsup_{t\to 0} \frac{ (\cF(\mu)-\cF(\mu_t))_+ }{W_2(\mu, \mu_t)} = \limsup_{t\to 0} \frac{ (f(0)-f(t))_+ }{D \cdot  t} = - \frac{f'_+(0)}D.$$
% Putting everything together $D^2 \leq- \tau f'_-(1) \leq -\tau (f'_+(0)+\lambda D^2) \leq \tau D \cdot(|\nabla^- \cF|(\mu)-\lambda D) $ thus $(1+\lambda \tau) D \leq \tau |\nabla^- \cF|(\mu)$ and we can conclude
% $$C_{\mu}^{\tau}=W_2(\mu,\prox_{\tau \cF}^W(\mu)) =D \leq  \frac{\tau}{1+\lambda \tau} |\nabla^- \cF|(\mu).$$
% \end{itemize}  
\end{proof}

%%%%%%%%%%%%%%%%%%%%%%%%%%%%%%%%%%%%%
%%%%%%%%%%%%%% Aknow %%%%%%%%%%%%%%%%%%
%%%%%%%%%%%%%%%%%%%%%%%%%%%%%%%%%
\section*{Acknowledgments}
The authors thank Alessandro Pinzi for pointing out the reference \cite{Be} and Alpár R. Mészáros and Filippo Santambrogio for \cite{Pa}.\\
The authors are members of GNAMPA at INdAM. The authors are supported by the INdAM-GNAMPA project “Geometria di spazi di misure: distanze di Wasserstein, loro varianti e applicazioni”, CUP E53C25002010001.  The authors acknowledge the support of  PRIN 202244A7YL. S. D. M. and E. N. acknowledge the financial support of the US Air Force Office of Scientific Research (FA8655-22-1-7034) and of the MUR Excellence Department Project awarded to Dipartimento di Matematica, Università di Genova, CUP D33C23001110001.   S. F. acknowledges the support
of the European Commission, grant TraDE-OPT 861137 and of the COST Action 24122 mSPACE, supported by COST (European Cooperation in Science and Technology), www.cost.eu. \\
This work represents only the view of the authors. The European Commission and the other organizations are not responsible for any use that may be made of the information it contains.

\bibliographystyle{abbrv}
\bibliography{bib}

%\section*{\textcolor{red}{Cose da Fare}}
%\begin{itemize}
    % \item SARA:\sout{mettere le notazioni in verticale}
    % \item SARA: \sout{aggiungere $3$-plan e altre cose di interpolazione nelle notazioni}
    % \item SIMO: \sout{uniformare $\mathcal{F}$ che è $\cF$}
    % \item SIMO: \sout{citazione per proiezione unica} Aggiunta la citazione, ma ci serve come lemma?
    % \item {\color{green}EMA: Citazione per contrattività o meno per $\lambda$ generali: \cite[Proposition 3.3]{HoheiselLabordeOberman2020}}
    % \item EMA: \sout{Controllare AGS e commenti sul dominio di G (nella generalized convexity)}
    % \item \sout{EMA: citazione continuità prox. HERE:} \cite[Theorem 4.1.2 (i)]{AGS08}
    % \item SIMO: \sout{Aggiungere $\lambda$-convexity nelle definizioni in notazioni}
    % \item SIMO: \sout{citare sara mm}
    % \item SIMO: \sout{sistemare l'ultima sezione}
    % \item SARA: \sout{cambiare il nome dell'insieme dei lower bounds controllare che è fatto dappertutto (in definizione ancora $\leq_C$)}
    % \item SARA: \sout{controllare il remark dopo l'esempio  entropia non convessa su 2-base gg}
    % \item SARA:\sout{scrivere la fine della proof che combaci con lo statement nell intro e cambiare lo statement}
    % \item SIMO: \sout{intro Sezione 3}
     %\item
     
     %Tutto in $\R^d$: cambiamo in Hilbert o mettiamo un remark che dice che è uguale? (controllare dove si usa la misura di Lebesgue)
%\end{itemize}
\end{document}